\documentclass[11pt,a4paper]{article}
\usepackage[utf8]{inputenc}
\usepackage[a4paper, total={6.7in, 9.7in}]{geometry}

\usepackage[T1]{fontenc}
\usepackage[variablett]{lmodern}
\usepackage{amsmath}
\usepackage{amsfonts}
\usepackage{amssymb}
\usepackage{float}
\usepackage{wrapfig}
\usepackage{enumitem}
\usepackage{wasysym}
\usepackage{graphicx}
\usepackage{tabularx}
\usepackage{quiver}

\usepackage[bottom]{footmisc}
\usepackage[backend=bibtex,style=alphabetic]{biblatex}
\bibliography{refs}

\usepackage{caption}
\usepackage{hyperref}
\hypersetup{
    colorlinks = true,
}
\PassOptionsToPackage{unicode}{hyperref}
\PassOptionsToPackage{naturalnames}{hyperref}
\usepackage[nameinlink]{cleveref}
\Crefname{subsection}{Subsection}{Subsections}
\Crefname{question}{Question}{Questions}
\Crefname{subsubsection}{Paragraph}{Paragraphs}

\usepackage{comment}

\newcommand{\customref}[2]{\hyperref[#2]{#1}}

\newcommand\iref[2]{\customref{\Cref*{#1}~\ref*{#2}}{#1}}

\newcommand{\Krull}{\textnormal{Krull}}
\newcommand{\Spm}{\textnormal{Spm}}

\newlength{\movedby}
\newcommand{\raisegroup}{\leavevmode\setlength{\abovedisplayskip}{0pt}\vspace{-\baselineskip}}

\makeatletter
\newcommand{\biggg}{\bBigg@\thr@@}
\newcommand{\Biggg}{\bBigg@{3.5}}
\newcommand{\vast}{\bBigg@{4}}
\newcommand{\Vast}{\bBigg@{5}}
\makeatother

\renewcommand{\emptyset}{\varnothing}

\newcommand{\N}{\mathbb{N}}
\newcommand{\Z}{\mathbb{Z}}
\newcommand{\Zpos}{\Z_{\geq 0}}
\newcommand{\Zsp}{\Z_{>0}}

\newcommand{\C}{\mathbb{C}}

\newcommand{\lcm}{\mathrm{lcm}}

\renewcommand{\Pr}{\mathbb{P}}
\newcommand{\PC}{\Pr^1(\C)}

\newcommand{\Oo}[1]{O\!\left(#1\right)}
\newcommand{\Os}[1]{O^{\sharp}\!\left(#1\right)}

\newcommand{\Spec}{\mathrm{Spec}}
\newcommand{\Proj}{\mathrm{Proj}}
\newcommand{\Frac}{\mathrm{Frac}}

\newcommand{\Sym}{\mathfrak{S}}

\newcommand{\card}[1]{ \left | #1 \right | }
\newcommand{\gen}[1]{ \left \langle #1 \right \rangle }
\newcommand{\ord}{\mathrm{ord}}

\newcommand{\B}{\mathrm{B}} 

\newcommand{\Comp}{\mathrm{Comp}} 

\newcommand{\id}{\mathrm{id}}
\newcommand{\Sub}{\mathrm{Sub}}
\renewcommand{\HF}{\mathrm{HF}}

\newcommand{\gbar}{\underline{g}}

\newcommand{\eqdef}{\overset{\text{def}}{=}}
\newcommand{\verti}{\, \middle \vert \,}

\newcommand{\Span}{\mathrm{Span}}

\renewcommand{\bar}{\overline}

\renewcommand{\tilde}{\widetilde}

\newcommand{\ab}{\textnormal{ab}}

\usepackage{amsthm}
\usepackage{thmtools}

\newcounter{mycounter}[section]

\theoremstyle{plain}
\newtheorem{theorem}[mycounter]{Theorem}
\newtheorem{corollary}[mycounter]{Corollary}
\newtheorem{proposition}[mycounter]{Proposition}
\newtheorem{lemma}[mycounter]{Lemma}

\theoremstyle{remark}
\newtheorem{remark}[mycounter]{Remark}

\theoremstyle{definition}
\newtheorem{definition}[mycounter]{Definition}

\newtheorem{example}[mycounter]{Example}

\counterwithin{equation}{section}

\usepackage{titletoc}
\titlecontents{section}
[0pt]                                               
{}%
{\contentsmargin{0pt}                               
    \large\thecontentslabel.\enspace
    }
{\contentsmargin{0pt}\large}                        
{\titlerule*[.5pc]{ }\contentspage}                 
[]                                                  

\usepackage{titlesec}
\titleformat{\section}[block]{\normalfont\centering\scshape\large}{\thesection.}{1em}{}
\titleformat{\subsection}[block]{\normalfont\large}{\thesubsection.}{1em}{\bf}
\titleformat{\subsubsection}[runin]{\normalfont}{\bf\thesubsubsection.}{0.3em}{\bf}

\makeatletter
\patchcmd{\@maketitle}{\LARGE}{\huge}{\typeout{OK 1}}{\typeout{Failed 1}}
\patchcmd{\@maketitle}{\large \lineskip}{\Large \lineskip}{\typeout{OK 2}}{\typeout{Failed 2}}
\makeatother

\title{
  The Geometry of Rings of Components \break of Hurwitz Spaces
}

\author{Béranger Seguin\footnote{Universität Paderborn, Fakultät EIM, Institut für Mathematik, Warburger Str. 100, 33098 Paderborn, Germany. Email: \texttt{bseguin@math.upb.de}.}}
\date{}

\renewenvironment{abstract}{%
\hfill\begin{minipage}{0.95\textwidth}
\rule{\textwidth}{1pt} \textsc{Abstract.}}
{\par\noindent\rule{\textwidth}{1pt}\end{minipage}}

\begin{document}

\maketitle{}

\begin{abstract}
	We consider a variant of the ring of components of Hurwitz spaces introduced by Ellenberg, Venkatesh and Westerland.
	By focusing on Hurwitz spaces classifying covers of the projective line, the resulting ring of components is commutative, which lets us study it from the point of view of algebraic geometry and relate its geometric properties to numerical invariants involved in our previously obtained asymptotic counts.
	Specifically, we describe a stratification of the prime spectrum of the ring of components, and we compute the dimensions and degrees of the strata.
	Using the stratification, we give a complete description of the spectrum in some cases.

	\bigskip

\end{abstract}

{
  \hypersetup{linkcolor=black}
  \tableofcontents{}
}
\hfill\rule{0.95\textwidth}{1pt}

\section{Introduction and main results}
\label{sn:intro}

For the whole article, we fix a finite group $G$, a nonempty set $D$ of nontrivial conjugacy classes of~$G$,
a map $\xi : D \to \Zsp$ (attributing a multiplicity to each conjugacy class $\gamma \in D$), and a field $k$ whose characteristic does not divide the order $\card{G}$ of the group $G$.

\subsection{Context}

In \cite{EVW}, Ellenberg, Venkatesh and Westerland introduced the \emph{ring of components of Hurwitz spaces}, a graded algebra whose elements are linear combinations of connected components of Hurwitz spaces parametrizing marked $G$-covers%
\footnote{
	Here, a marked $G$-cover is a finite branched cover (not necessarily connected) with a marked point in an unramified fiber, equipped with an action of $G$ on the cover inducing simply transitive actions of $G$ on each unramified fiber.
}
of the affine line.
The grading of that ring reflects the number of branch points of the covers parametrized by each component, and the multiplicative structure is induced by a geometric ``concatenation'' operation.

The definition of that ring is motivated by the fact that its Hilbert function is tightly related to the asymptotic behavior of the cohomology of Hurwitz spaces, which is in turn related (using the Grothendieck-Lefschetz trace formula and Deligne's bounds on the eigenvalues of Frobenius endomorphisms) to the count of $\mathbb{F}_q$-points of Hurwitz spaces and hence to the distribution of extensions $F|\mathbb{F}_q(T)$ with Galois group isomorphic to~$G$, when $q$ is large and coprime to $\card{G}$.
In \cite{ETW}, this approach was used to obtain an upper bound consistent with the variant of Malle's conjecture for function fields over finite fields.

In \cite{countcomp}, we have extended some of the counting results of \cite{EVW}.
For instance, we have studied the analogous ring of components of Hurwitz spaces of marked $G$-covers of the \emph{projective} line.
This ring is a \emph{commutative} graded finitely generated algebra, and the growth of its Hilbert function is related to geometric invariants of its spectrum.
This observation was the starting point for a more systematic study of the ring of components from the point of view of algebraic geometry.

\subsection{Main results}

In \Cref{sn:def-prelim}, we define the \emph{ring of components} $R$ (\Cref{defn:ringcomp}), which is a finitely generated commutative graded $k$-algebra.
We then introduce its prime spectrum $\Spec\,R$, which we call the \emph{variety of components} (\Cref{defn:varcomp}).
In \Cref{sn:stratification}, we define subsets $\gamma(H)$ of $\Spec\,R$ (\Cref{defn:stratum}), indexed by subgroups $H$ of $G$, and we prove that they form a stratification of the variety of components:

\begin{theorem}
	\label{prop:stratification-of-spec}
	The locally closed subsets $\gamma(H)$ form a stratification of $\Spec\,R$:
	\[
		\Spec\,R
		=
		\bigsqcup_{H \subseteq G}
			\gamma(H).
	\]
\end{theorem}

This result, which is a particular case of the more general \Cref{thm:z-ih-is-made-of-gamma}, has the following consequence: in order to describe the variety of components fully, it suffices to describe each stratum $\gamma(H)$.
Using the counting results of \cite{countcomp}, we compute in \Cref{sn:proof-dim-gamma-eq-zeta} the Krull dimension of the stratum~$\gamma(H)$ corresponding to a subgroup $H$ of $G$.
More precisely, we relate it to a numerical invariant defined in \cite{countcomp}, the \emph{splitting number} $\Omega(D_H)$ (\Cref{defn:splitting-number}):

\begin{theorem}
	\label{thm:dim-gamma-omega}
	We have $\dim_\Krull \gamma(H) = \Omega(D_H) + 1$.
\end{theorem}

In \Cref{subsn:degree-gamma}, we discuss further connections between group-theoretic and geometric invariants by relating the degree of the stratum $\gamma(H)$, seen as embedded in projective space, to (a quotient of) the second homology group of $H$.

In \Cref{sn:spectrum-2}, we approach the variety of components more ``directly'' by describing the strata fully in \Cref{thm:desc-spectrum} and its coordinate-based variant \Cref{thm:desc-spectrum-coords}.
However, our description relies on strong assumptions on the ring of components.
We do not reproduce the statement here as it uses a lot of terminology.
This result applies in particular to the classical situation where $G$ is a symmetric group and $D$ contains only the conjugacy class of transpositions.
In that case, \Cref{thm:varcomp-symgp} gives a full description of the variety of components.

\subsection{Outline}

This article is organized as follows:

\begin{itemize}
	\item
		In \Cref{sn:def-prelim}, we define notation and terminology used throughout the article.
		Notably, we define the ring of components (\Cref{defn:ringcomp}) and its associated variety (\Cref{defn:varcomp}), which are our main objects of study.
	\item
		In \Cref{sn:stratification}, we associate to each subgroup $H$ of $G$ a subring $R^H$ and four ideals $I_H, I^*_H, J_H, J^*_H$ of the ring of components (\Cref{defn:important-subspaces}).
		We use these to define the strata $\gamma(H)$ (\Cref{defn:stratum}).
		We then prove \Cref{thm:z-ih-is-made-of-gamma}, which is the general form of the stratification of the variety of components (\Cref{prop:stratification-of-spec}).
	\item
		In \Cref{sn:ringcomp-nearly-reduced}, we prove \Cref{thm:ringcomp-nearly-reduced}.
		This technical result, which is a weak asymptotic form of reducedness for the ring of components $R$, is needed for the proof of \Cref{thm:dim-gamma-omega}.
	\item
		In \Cref{sn:proof-dim-gamma-eq-zeta}, we compute the Krull dimension of each stratum $\gamma(H)$ (\Cref{thm:dim-gamma-omega}).
		In \Cref{subsn:degree-gamma}, we also compute the degree of $\gamma(H)$ in some cases.
		The proofs rely on the asymptotic counting results from \cite{countcomp}.
	\item
		In \Cref{sn:spectrum-2}, we prove \Cref{thm:desc-spectrum,thm:desc-spectrum-coords}, which give complete descriptions of the variety of components in some cases.
		We apply these results to the classical case of symmetric groups in \Cref{ssn:symgp}.
\end{itemize}

\subsection{Acknowledgments}

This work was funded by the French ministry of research through a CDSN grant, and by the Deutsche Forschungsgemeinschaft (DFG, German Research Foundation) --- Project-ID 491392403 --- TRR 358.

I thank my advisors Pierre Dèbes and Ariane Mézard for their support and their precious advice during my time as a PhD student, and the reporters of my thesis Jean-Marc Couveignes and Craig Westerland for providing helpful feedback.

\section{Definitions and preliminaries}
\label{sn:def-prelim}

Recall that we have fixed a finite group $G$, a set $D$ of nontrivial conjugacy classes of $G$, and a map $\xi : D \to \Zsp$.
Additionally, we define the set $c = \bigsqcup_{\gamma \in D} \gamma$ and the integer  $\card{\xi} = \sum_{\gamma \in D} \xi(\gamma)$.

\subsection{The monoid of components}
\label{subsn:moncomp}

We briefly recall the definition of the \emph{monoid of components} $\Comp_{\PC}(G , \, D , \, \xi)$, which was already defined in \cite[Definition~3.4.4]{SegThese} and \cite[Definition~2.6]{countcomp}.
First, we define the braid group $\B_n$ by its presentation:

\begin{definition}
	\label{defn:braid}
	The \emph{Artin braid group} $\B_n$ on $n$ strands is defined by the following presentation:
	\[
		\B_n
		\eqdef
		\left\langle
			\sigma_1, \,
			\sigma_2, \,
			\ldots, \,
			\sigma_{n-1}
			\verti
			\begin{array}{cccc}
				\sigma_i \sigma_j
				& =
				& \sigma_j \sigma_i
				& \text{if } |i-j| > 1
				\\
				\sigma_i \sigma_{i+1} \sigma_i
				& =
				& \sigma_{i+1} \sigma_i \sigma_{i+1}
				& \text{if } i \in \{1, \ldots, n-2\}
			\end{array}
		\right\rangle.
	\]
\end{definition}

\begin{definition}
	The \emph{Hurwitz action} of $\B_n$ on the set $G^n$ of $n$-tuples of elements of $G$ is the (well-defined) action for which the generator $\sigma_i \in \B_n$ acts on a tuple $\gbar = (g_1, \ldots, g_n) \in G^n$ as follows:
	\[
		\sigma_i .
		(
			g_1, \,
			\ldots, \,
			g_{i-1}, \,
			g_i, \,
			g_{i+1}, \,
			g_{i+2}, \,
			\ldots, \,
			g_n
		)
		=
		(
			g_1, \,
			\ldots, \,
			g_{i-1}, \,
			g_i g_{i+1} g_i^{-1}, \,
			g_i, \,
			g_{i+2}, \,
			\ldots, \,
			g_n
		).
	\]
\end{definition}

\begin{definition}
	\label{defn:main-invariants}
	Let $\gbar = (g_1, \ldots, g_n) \in G^n$ be a tuple of elements of $G$.
	The \emph{group} of $\gbar$ is the subgroup~$\gen{\gbar}$ of $G$ generated by $g_1, \ldots, g_n$, and the \emph{product} of $\gbar$ is the element $\pi\gbar \eqdef g_1 \cdots g_n \in G$.
\end{definition}

Both the group and product of a tuple are invariant under the Hurwitz action, and thus we extend the definition of these invariants and the notations $\gen{m}, \pi m$ when $m$ is an orbit for the Hurwitz action.

\begin{definition}
	\label{defn:moncomp}
	A \emph{component} (of degree $n$) is the orbit, under the Hurwitz action of the braid group~$\B_{n \card{\xi}}$, of a tuple $\gbar = (g_1, \ldots, g_{n\card{\xi}}) \in G^{n\card{\xi}}$ satisfying $\pi\gbar = 1$ and such that exactly $n \cdot \xi(\gamma)$ entries of $\gbar$ belong to each conjugacy class $\gamma \in D$.
	The \emph{monoid of components} $\Comp_{\PC}(G , \, D , \, \xi)$ is the (nonnegatively) graded set whose elements of degree $n$ are the components of degree $n$, equipped with the multiplication induced by the concatenation of tuples:
	\[
		(
			g_1, \,
			\ldots, \,
			g_{n\card{\xi}}
		)
		(
			g'_1, \,
			\ldots, \,
			g'_{n'\card{\xi}}
		)
		=
		(
			g_1, \,
			\ldots, \,
			g_{n\card{\xi}}, \,
			g'_1, \,
			\ldots, \,
			g'_{n\card{\xi}}
		).
	\]
\end{definition}

The monoid of components is well-defined and commutative (\cite[Proposition~3.3.8]{SegThese}, \cite[Proposition~3.3.11]{SegThese}).
Components of degree $n$ are named this way because they correspond bijectively to connected components of the Hurwitz space classifying marked $G$-covers of the projective line branched at $n \cdot \card{\xi}$ points, among which $n \cdot \xi(\gamma)$ have their monodromy elements in each class $\gamma \in D$.
This connection is explained more carefully in \cite[Subsection~3.3.2]{SegThese}.
The identity element of the monoid $\Comp_{\PC}(G , \, D , \, \xi)$ is the orbit of the empty tuple, which corresponds to the connected component containing only the trivial $G$-cover (with no branch points).

\begin{definition}
	\label{defn:nonfact}
	A nontrivial element of $\Comp_{\PC}(G , \, D , \, \xi)$ is a \emph{non-factorizable component}%
	\footnote{
		Non-factorizable components are simply the irreducible elements of the monoid $\Comp_{\PC}(G , \, D , \, \xi)$, but we avoid using the ambiguous term ``irreducible component''.
	}
	if it does not equal any product of two nontrivial components.
\end{definition}

A simple pigeonhole argument (carried out in \cite[Lemma~3.4.17]{SegThese}) shows that there are finitely many non-factorizable components.
Therefore, the monoid of components is a finitely generated commutative graded monoid.

\begin{remark}
	\label{rk:not-same-degree}
	The non-factorizable components do not necessarily all have the same degree; see \cite[Remark~3.4.20]{SegThese} for a counterexample.
\end{remark}

\subsection{The ring of components}
\label{subsn:ringcomp}

\subsubsection{Definition.}

We now define the ring of components:
\begin{definition}
	\label{defn:ringcomp}
	The \emph{ring of components} $R$ is the graded $k$-algebra $k[\Comp_{\PC}(G , \, D , \, \xi)]$ obtained as the monoid ring (over $k$) of the monoid of components.
	The \emph{irrelevant ideal} $\varpi$ is the (maximal) ideal of~$R$ generated by components of positive degree.
\end{definition}

The ring $R$ of \Cref{defn:ringcomp} corresponds to $R_{\PC}(G,\, D,\, \xi)$ in the notation of \cite[Definition~3.4.12]{SegThese}.
The properties mentioned in \Cref{subsn:moncomp} imply that the ring $R$ is a commutative graded $k$-algebra of finite type, generated by the non-factorizable components.

\subsubsection{Variety of components.}

We now define our main object of study, the \emph{variety of components}:

\begin{definition}
	\label{defn:varcomp}
	The \emph{variety of components} is the set $\Spec\,R$ of prime ideals $\mathfrak{p} \subsetneq R$, equipped with the Zariski topology.
	If $I$ is an ideal of $R$, we denote by $V(I)$ the closed subset of $\Spec\,R$ consisting of all prime ideals containing $I$.
\end{definition}

\subsubsection{Affine embedding.}
\label{subsubsn:affine-embed}

Assume that $k$ is algebraically closed, and let $\Spm\,R$ be the subset of $\Spec\,R$ consisting of closed points, i.e., of maximal ideals $\mathfrak{m} \subsetneq R$.
Then, the set $\Spm\,R$ can be identified with the set of morphisms of $k$-algebras from $R$ to $k$ (identifying a maximal ideal~$\mathfrak{m}$ with the projection $R \twoheadrightarrow R/\mathfrak{m} \simeq k$), or equivalently with the set of $k$-points of the scheme~$\Spec\,R$.
Let~$\Sigma$ be the finite set of non-factorizable components.
Then, we can identify $\Spm\,R$ with a classical variety by embedding it in~$k^\Sigma$ as follows: a point $(x_m)_{m \in \Sigma}$ belongs to $\Spm\,R$ if and only if the equality $x_{m_1} \cdots x_{m_u} = x_{m'_1} \cdots x_{m'_v}$ holds whenever the equality $m_1 \cdots m_u = m'_1 \cdots m'_v$ holds in the monoid of components.
Note that, as the monoid of components is commutative and finitely generated, Dickson's lemma implies that it is presented by finitely many equalities of that type.

\begin{remark}
	\label{rk:spec-not-proj}
	The ring $R$ is a graded $k$-algebra, and thus it may be more natural to consider the projective variety that it defines (i.e., the set of homogeneous ideals which are maximal among those properly contained in~$\varpi$) instead of the affine variety $\Spm\,R$.
	However, since non-factorizable components need not all have the same degree (cf. \Cref{rk:not-same-degree}), the space in which the variety naturally embeds is a ``weighted'' projective space, namely the set of orbits of $k^\Sigma \setminus \{0\}$ under the action of~$k^{\times}$ for which a scalar $\lambda \in k^\times$ acts on a point $z$ by multiplying its coordinate $z_m$ (associated to a non-factorizable component $m \in \Sigma$) by $\lambda^{\deg m}$.
	Weighted projective spaces do embed in ordinary projective spaces of higher dimension \autocite[Theorem~3.4.9]{hosgood}, but we mostly work with the affine variety associated to $R$ to avoid dealing with these subtleties.
\end{remark}

\subsection{$D$-generated subgroups}

We briefly recall the notion of \emph{$D$-generated subgroups} from \cite[Definition~1.1]{countcomp}:

\begin{definition}
	\label{defn:dgen}
	A subgroup $H$ of $G$ is $D$-generated if the sets $\gamma \cap H$ for $\gamma \in D$ are all nonempty and collectively generate~$H$.
	We denote by $\Sub_{G,D}$ the set of subgroups of $G$ which are either trivial or $D$-generated.
\end{definition}

The relevance of this definition comes from the following proposition, which is proved in \cite[Proposition~3.2.22]{SegThese}:

\begin{proposition}
	\label{prop:dgen-are-monodromy-groups}
	A subgroup $H$ of $G$ belongs to $\Sub_{G,D}$ if and only if there is a component $m \in \Comp_{\PC}(G , \, D , \, \xi)$ whose group is $H$.
\end{proposition}

If $H$ is a $D$-generated subgroup of $G$, we define its \emph{splitting number} as in \cite[Definition~1.2]{countcomp}:

\begin{definition}
	\label{defn:splitting-number}
	Let $H \in \Sub_{G,D}$.
	Let $D_H \eqdef \left\{ \gamma \cap H \verti \gamma \in D \right\}$, let $c_H \eqdef c \cap H$ (which is also $\bigsqcup_{\gamma' \in D_H} \gamma'$), and denote by $D^*_H$ the set of conjugacy classes of $H$ which are contained in $c_H$.
	The \emph{splitting number} of $H$ is the integer $\Omega(D_H) \eqdef \card{D^*_H} - \card{D_H}$.
\end{definition}

\begin{definition}
	A $D$-generated subgroup $H$ of $G$ is a \emph{non-splitter} if $\Omega(D_H) = 0$, i.e., if $D_H$ consists of conjugacy classes of $H$.
\end{definition}

The splitting number of $H$ plays a central role in the asymptotic count of components with group $H$, cf. \cite[Theorem~1.4]{countcomp}.

\subsection{Chart of notations}
\label{subsn:char-not}

For quick reference, the chart below indicates where the definitions introduced in this section can be found.
A short description is also given.

\medskip
{
	\centering
	\begin{tabularx}{\textwidth}{
		|>{\centering\arraybackslash}p{4cm}
		|>{\centering\arraybackslash}p{3cm}
		|>{\centering\arraybackslash}X|}
		\hline
		\bf Notation &
		\bf Reference &
		\bf Short description
		\\
		\hline
		$G,D,\xi,k$ &
		Top of \Cref{sn:intro} &
		setup
		\\
		$c, \card{\xi}$ &
		Top of \Cref{sn:def-prelim} &
		\\
		$\B_n$ &
		\Cref{defn:braid} &
		Artin braid group
		\\
		$\gen{\gbar}, \pi\gbar$ &
		\Cref{defn:main-invariants} &
		invariants of a tuple (or component)
		\\
		$\Comp_{\PC}(G , \, D , \, \xi)$ &
		\Cref{defn:moncomp} &
		monoid of components
		\\
		$R$ &
		\Cref{defn:ringcomp} &
		ring of components
		\\
		$\Spec\,R, V(I)$ &
		\Cref{defn:varcomp} &
		variety of components and its closed subsets
		\\
		$\Sub_{G,D}$ &
		\Cref{defn:dgen} &
		set of $D$-generated (or trivial) subgroups
		\\
		$\Omega(D_H)$ &
		\Cref{defn:splitting-number} &
		splitting number of $H$
		\\
		\hline
	\end{tabularx}
}

\medskip

We also include a chart of notation introduced in later sections:

\medskip
{
	\centering
	\begin{tabularx}{\textwidth}{
		|>{\centering\arraybackslash}p{4cm}
		|>{\centering\arraybackslash}p{3cm}
		|>{\centering\arraybackslash}X|}
		\hline
		\bf Notation &
		\bf Reference &
		\bf Short description
		\\
		\hline
		$I_H, I^*_H, J_H, J^*_H$ &
		\Cref{defn:important-subspaces} &
		ideals of $R$ associated to a subgroup $H$
		\\
		$R^H$ &
		\Cref{defn:important-subspaces} &
		subring of $R$ associated to a subgroup $H$
		\\
		$\gamma(H)$ &
		\Cref{defn:stratum} &
		stratum associated to a subgroup $H$
		\\
		$\Gamma_H$ &
		\Cref{defn:ideal-gamma} &
		ideal quotient $(\sqrt{I^*_H}:\sqrt{I_H})$
		\\
		$R_{n,H}, N_{n,H}$
		&
		Top of
		\Cref{sn:ringcomp-nearly-reduced}
		&
		space spanned by components of degree $n$ and group $H$ (resp. subspace of nilpotent elements)
		\\
		$\mu_H, H_2(H, c_H)$
		&
		\Cref{subsn:lifting-invariant}
		&
		notation related to the lifting invariant
		\\
		\hline
	\end{tabularx}
}

\section{The subgroup stratification of the variety of components}
\label{sn:stratification}

In this section, we establish the stratification of the variety of components (\Cref{prop:stratification-of-spec}).
To this end, we prove the more general \Cref{thm:z-ih-is-made-of-gamma}.

\subsection{Remarkable ideals and subrings of the ring of components}

\begin{definition}
	\label{defn:important-subspaces}
	Let $H$ be a subgroup of $G$.
	We define the following graded $k$-linear subspaces of the ring of components~$R$, each of which is the space spanned by components $m \in \Comp_{\PC}(G , \, D , \, \xi)$ (\Cref{defn:moncomp}) whose group $\gen{m}$ (\Cref{defn:main-invariants}) satisfies some condition:
	\begin{itemize}
		\item
			$I_H$ is spanned by components whose group contains $H$.
		\item
			$I^*_H$ is spanned by components whose group properly contains $H$.
		\item
			$J_H$ is spanned by components whose group is not properly contained in $H$.
		\item
			$J^*_H$ is spanned by components whose group is not contained in $H$.
		\item
			$R^H$ is spanned by components whose group is contained in $H$.
	\end{itemize}
\end{definition}

For any subgroup $H \subseteq G$, the subspaces $I_H$, $I^*_H$, $J_H$ and $J^*_H$ are homogeneous ideals of $R$, and~$R^H$ is a graded subalgebra of $R$, isomorphic to the quotient $R/J^*_H$ (\cite[Proposition~3.4.26]{SegThese}).
Note that, by \Cref{prop:dgen-are-monodromy-groups}, the inclusions $I^*_H \subseteq I_H$ and $J^*_H \subseteq J_H$ are strict if and only if $H \in \Sub_{G,D}$.
The following properties are straightforward to check (cf. \cite[Proposition~3.4.24]{SegThese}):

\begin{proposition}
	\label{prop:properties-of-ih-jh}
	The subspaces defined in \Cref{defn:important-subspaces} satisfy the following properties:
	\begin{enumerate}[label=(\roman*)]
		\item
			Let $H \subseteq H'$ be subgroups of $G$.
			Then, we have the inclusion $R^H \subseteq R^{H'}$ between the associated subalgebras of $R$, and ``mirrored'' inclusions between the associated ideals of $R$:
			\[
				I_{H'} \subseteq I_H
				\hspace{2cm}
				I^*_{H'} \subseteq I^*_H
				\hspace{2cm}
				J_{H'} \subseteq J_H
				\hspace{2cm}
				J^*_{H'} \subseteq J^*_H.
			\]
		\item
			We have:
			\[
				I_1 = J_1 = R^G = R
				\hspace{2cm}
				I^*_1 = J^*_1 = \varpi
				\hspace{2cm}
				I^*_G = J^*_G = 0
				\hspace{2cm}
				R^1 = k.
			\]
		\item
			For every subgroup $H \subseteq G$, we have:
			\[
				I^*_H = I_H \cap J^*_H
				\hspace{2cm}
				J_H = I_H + J^*_H.
			\]
		\item
			For every subgroup $H \subseteq G$, we have:
			\[
				I^*_H
				=
				\sum_{H' \supsetneq H}
					I_{H'}
				\hspace{2cm}
					J_H
				=
				\bigcap_{H' \subsetneq H}
					J^*_{H'}.
			\]
		\item
			For all subgroups $H_1, H_2 \subseteq G$, we have:
			\[
				I_{\gen{H_1,H_2}}
				=
				I_{H_1} \cap I_{H_2}
				\hspace{2cm}
				J^*_{H_1 \cap H_2}
				=
				J^*_{H_1} + J^*_{H_2}.
			\] 
	\end{enumerate}
\end{proposition}

\subsection{Remarkable subsets of the variety of components}

We rephrase some of the properties from \Cref{prop:properties-of-ih-jh} geometrically:
\begin{proposition}
	\label{prop:properties-Zih}
	The closed subsets $V(I_H)$, $V(I^*_H)$, $V(J_H)$ and $V(J^*_H)$ of $\Spec\,R$, defined using the ideals from \Cref{defn:important-subspaces}, satisfy the following properties:
	\begin{enumerate}[label=(\roman*)]
		\item
			\label{prop:properties-Zih-2}
			Let $H \subseteq H'$ be subgroups of $G$.
			Then, the following inclusions hold:
			\[
				V(I_H) \subseteq V(I_{H'})
				\hspace{1.5cm}
				V(I^*_H) \subseteq V(I^*_{H'})
				\hspace{1.5cm}
				V(J_H) \subseteq V(J_{H'})
				\hspace{1.5cm}
				V(J^*_H) \subseteq V(J^*_{H'}).
			\]
		\item
			\label{prop:properties-Zih-ii}
			We have:
			\[
				V(I_1) = V(J_1) = \emptyset
				\hspace{2cm}
				V(I^*_1) = V(J^*_1) = \{\varpi\}
				\hspace{2cm}
				V(I^*_G) = V(J^*_G) = \Spec\,R.
			\]
		\item
			\label{prop:properties-Zih-4}
			For every subgroup $H \subseteq G$, we have:
			\[
				V(I^*_H) = V(I_H) \cup V(J^*_H)
				\hspace{2cm}
				V(J_H) = V(I_H) \cap V(J^*_H).
			\]

		\item
			\label{prop:properties-Zih-5}
			For every subgroup $H \subseteq G$, we have:
			\[
				V(I^*_H)
				=
				\bigcap_{
					H' \supsetneq H
				}
					V(I_{H'})
				\hspace{2cm}
				V(J_H)
				=
				\bigcup_{
					H' \subsetneq H
				}
					V(J^*_{H'}).
			\]

		\item
			\label{prop:properties-Zih-6}
			For all subgroups $H_1, H_2 \subseteq G$, we have:
			\[
				V(I_{\gen{H_1,H_2}}) = V(I_{H_1}) \cup V(I_{H_2})
				\hspace{2cm}
				V(J^*_{H_1 \cap H_2}) = V(J^*_{H_1}) \cap V(J^*_{H_2}).
			\] 
	\end{enumerate}
\end{proposition}

\subsection{The strata $\gamma(H)$}

We fix a subgroup $H$ of $G$.

\begin{definition}
	\label{defn:stratum}
	The \emph{stratum associated to $H$} is the (locally closed) subset
	$
		\gamma(H)
		\eqdef
		V(I^*_H) \setminus V(I_H)
	$
	of~$\Spec\,R$, i.e., the set of prime ideals of $R$ which contain $I^*_H$ but do not contain $I_H$.
\end{definition}

In other words, a prime ideal $p$ of~$R$ belongs to $\gamma(H)$ if and only if $p$ contains every component whose group is strictly larger than $H$, but there is a component $m$ of group exactly $H$ which does not belong to $p$.
The stratum $\gamma(H)$ is nonempty if and only if $H \in \Sub_{G,D}$.
Note that $\gamma(1) = \{\varpi\}$ by \iref{prop:properties-Zih}{prop:properties-Zih-ii}.

\begin{proposition}
	\label{prop:gamma-with-J}
	The stratum $\gamma(H)$ is equal to $V(J^*_H) \setminus V(J_H)$.
\end{proposition}

\begin{proof}
	\raisegroup
	\begin{align*}
		\gamma(H)
		& =
		V(I^*_H) \setminus V(I_H)
		&
		\text{by \Cref{defn:stratum}}
		\\
		& =
		\Big(
			V(I_H) \cup V(J^*_H)
		\Big)
		\setminus
		V(I_H)
		&
		\text{by \iref{prop:properties-Zih}{prop:properties-Zih-4}}
		\\
		& =
		V(J^*_H)
		\setminus
		\Big(
			V(I_H) \cap V(J^*_H)
		\Big)
		\\
		& =
		V(J^*_H) \setminus V(J_H)
		&
		\text{by \iref{prop:properties-Zih}{prop:properties-Zih-4}}.
		&\qedhere
	\end{align*}
\end{proof}

The relations between $V(I^*_H), V(I_H), V(J^*_H), V(J_H)$ and $\gamma(H)$ are summarized by the following diagram:
\begin{center}
	\includegraphics[scale=.5]{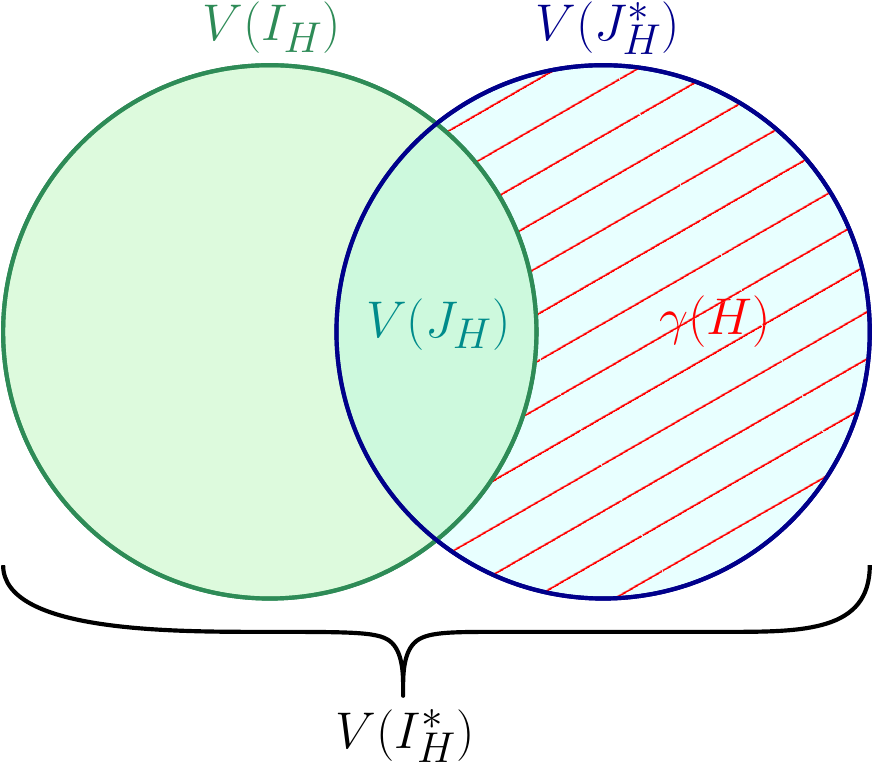}
\end{center}

In general, the strata are neither closed nor open subsets of $\Spec\,R$, and as such they not have obvious algebraic counterparts.
However, their closure in $\Spec\,R$ can be described using ideal quotients:

\begin{definition}
	\label{defn:ideal-gamma}
	We define the subset $\Gamma_H \subseteq R$ as the ideal quotient $(\sqrt{I^*_H}:\sqrt{I_H})$:
	\[
		\Gamma_H
		\eqdef
		(\sqrt{I^*_H}:\sqrt{I_H})
		=
		\left\lbrace
			x \in R
			\verti
			x \sqrt{I_H} \subseteq \sqrt{I^*_H}
		\right\rbrace
		=
		\left\lbrace
			x \in R
			\verti
			(x) \cap \sqrt{I_H} \subseteq \sqrt{I^*_H}
		\right\rbrace
		.
	\]
	(The last equality comes from the fact that a product of two ideals has the same radical as their intersection.)
\end{definition}

\begin{proposition}
	\label{prop:GammaH}
	The set $\Gamma_H$ satisfies the following properties:
	\begin{enumerate}[label=(\roman*)]
		\item
			\label{prop:GammaH-i}
			$\Gamma_H$ is a homogeneous radical ideal of $R$, namely the intersection of all prime ideals $p \in \gamma(H)$.
		\item
			\label{prop:GammaH-ii}
			The closed subset $V(\Gamma_H) \subseteq \Spec\,R$ associated to $\Gamma_H$ is the closure of $V(I_H) \setminus V(I^*_H) = \gamma(H)$.
		\item
			\label{prop:GammaH-iii}
			$\dim_\Krull R/\Gamma_H = \dim_\Krull \overline{\gamma(H)} = \dim_\Krull \gamma(H)$.
		\item
			\label{prop:GammaH-iv}
			$\Gamma_H \cap \sqrt{I_H} = \sqrt{I^*_H}$.
		\item
			\label{prop:GammaH-v}
			An ideal $I$ of $R$ satisfies $I \cap \sqrt{I_H} \subseteq \sqrt{I^*_H}$ if and only if $I$ is contained in $\Gamma_H$.
		\item
			\label{prop:GammaH-vi}
			Any ideal $I$ of $R$ satisfying $I \cap I_H \subseteq \Gamma_H$ is contained in $\Gamma_H$.
	\end{enumerate}
\end{proposition}

\begin{proof}
	Points \ref{prop:GammaH-i} and \ref{prop:GammaH-ii} follow from well-known properties of ideal quotients, see for example \cite[Chapter~4, §4]{coxlittleoshea}.
	Point \ref{prop:GammaH-iii} follows from \ref{prop:GammaH-ii} (for the first equality) and \autocite[Proposition I.1.10]{hartshorne} (for the second equality).
	Points \ref{prop:GammaH-iv} and \ref{prop:GammaH-v} follow directly from the definition of $\Gamma_H$.
	For point \ref{prop:GammaH-vi}, assume that~$I$ is an ideal such that $I \cap I_H \subseteq \Gamma_H$.
	Then, $I \cap \sqrt{I_H} \subseteq \sqrt{I \cap I_H} \subseteq \Gamma_H \cap \sqrt{I_H} = \sqrt{I^*_H}$, where the last equality  is point \ref{prop:GammaH-iv}.
	By point \ref{prop:GammaH-v}, we then have $I \subseteq \Gamma_H$.
\end{proof}

\subsection{The subgroup stratification of $\Spec\,R$}

In this subsection, we show that the strata $\gamma(H)$ defined in \Cref{defn:stratum} form a stratification of the variety of components.
We first show that the strata are disjoint:

\begin{proposition}
	\label{prop:strata-disjoint}
	Let $H, H'$ be two distinct subgroups of $G$.
	Then $\gamma(H) \cap \gamma(H') = \emptyset$.
\end{proposition}

\begin{proof}
	If $H \subsetneq H'$, then the set $\gamma(H)$ is contained in $V(I^*_H)$ and hence in $V(I_{H'})$ by \iref{prop:properties-Zih}{prop:properties-Zih-5}, and thus $\gamma(H')$ has an empty intersection with~$\gamma(H)$.
	The case $H' \subsetneq H$ is symmetric.
	We now assume that neither of $H$ and $H'$ is contained in the other, so that $\tilde H = \gen{H,\, H'}$ is a subgroup of $G$ strictly larger than both~$H$ and $H'$.
	We let $J = \gamma(H) \cap \gamma(H')$.
	By \iref{prop:properties-Zih}{prop:properties-Zih-5}, we have $V(I^*_H) \subseteq V(I_{\tilde H})$ and hence $\gamma(H) \subseteq V(I_{\tilde H})$.
	The same holds for $H'$ and thus $J \subseteq V(I_{\tilde H})$.
	But then $V(I_{\tilde H}) \setminus J$ is a subset of $V(I_{\tilde H})$ containing both $V(I_H)$ and $V(I_{H'})$.
	By \iref{prop:properties-Zih}{prop:properties-Zih-6}, we have $V(I_{\tilde H}) = V(I_H) \cup V(I_{H'})$ which finally implies $J=\emptyset$.
\end{proof}

We now prove the general form of the stratification:
\begin{theorem}
	\label{thm:z-ih-is-made-of-gamma}
	For every subgroup $H \subseteq G$, we have the following equalities:%
	{\everymath={\displaystyle}%
	\[\begin{array}{ccc}
		V(I_H) =
		\bigsqcup_{\substack{
			H' \textnormal{ subgroup of } G \\
			\textnormal{not containing } H
		}}
			\gamma(H')
		& &
		V(I^*_H) =
		\bigsqcup_{\substack{
			H' \textnormal{ subgroup of } G \\
			\textnormal{not properly containing } H
		}}
			\gamma(H')
		\\
		\vspace{0.1cm}
		&
		\phantom{aaaaaa}
		&
		\\
		V(J_H) =
		\bigsqcup_{H' \subsetneq H}
			\gamma(H')
		& &
		V(J^*_H) =
		\bigsqcup_{H' \subseteq H}
			\gamma(H').
	\end{array}\]
	}
\end{theorem}

By setting $H=G$ in the equality concerning $V(I^*_H)$, \Cref{thm:z-ih-is-made-of-gamma} implies \Cref{prop:stratification-of-spec}.

\begin{proof}
	We show by decreasing induction on the size of a subgroup $H$ of $G$ that:
	\begin{equation}
		\label{eqn:rechyp-stratif}
		\Spec\,R \setminus V(I_H)
		=
		\bigcup_{H' \supseteq H}
			\gamma(H').
	\end{equation}
	Let $H$ be a subgroup of $G$, and assume that every subgroup $H' \subseteq G$ with $\card{H'}>\card{H}$ satisfies \Cref{eqn:rechyp-stratif}.
	Then:
	\begin{align*}
		\Spec\,R
		\setminus
		V(I^*_H)
		& =
		\Spec\,R
		\setminus
		\bigcap_{H' \supsetneq H}
			V(I_{H'})
		&
		\text{by \iref{prop:properties-Zih}{prop:properties-Zih-5}}
		\\
		& =
		\bigcup_{H' \supsetneq H}
			\Big(
				\Spec\,R \setminus
				V(I_{H'})
			\Big)
		\\
		& =
		\bigcup_{H' \supsetneq H}
			\,
			\bigcup_{H'' \supseteq H'}
				\gamma(H'')
		&
		\text{by induction hypothesis}
		\\
		& =
		\bigcup_{H' \supsetneq H}
			\gamma(H').
	\end{align*}
	We conclude the induction by using the equality $\Spec\,R \setminus V(I_H) = \gamma(H) \cup (\Spec\,R \setminus V(I^*_H))$, which follows directly from \Cref{defn:stratum}.
	The union in \Cref{eqn:rechyp-stratif} is a disjoint union by \Cref{prop:strata-disjoint}.
	Plugging in $H=1$ yields:
	\begin{equation}
		\label{eqn:partition-of-V}
		\Spec\,R
		=
		\Spec\,R \setminus V(I_1)
		=
		\bigsqcup_{H' \subseteq G}
			\gamma(H')
	\end{equation}
	which is \Cref{prop:stratification-of-spec}.
	Finally:
	\begin{align*}
		V(I_H)
		& =
		\Spec\,R \setminus
		\Big(
			\Spec\,R \setminus
			V(I_H)
		\Big)
		\\
		& =
		\left(
			\bigsqcup_{H' \subseteq G}
				\gamma(H')
		\right)
		\setminus
		\left(
			\bigsqcup_{H' \supseteq H}
				\gamma(H')
		\right)
		&
		\text{
			by
			\Cref{eqn:partition-of-V,eqn:rechyp-stratif}
		}
		\\
		& =
		\bigsqcup_{\substack{
			H' \textnormal{ subgroup of } G \\
			\textnormal{not containing } H
		}}
			\gamma(H').
	\end{align*}
	Since $V(I^*_H) = V(I_H) \cup \gamma(H)$ (\Cref{defn:stratum}), we have the corresponding formula for $V(I^*_H)$.

	Dually, we show by increasing induction on the size of a subgroup $H \subseteq G$ that:
	\begin{equation}
		\label{eqn:rechyp-stratif-2}
		V(J^*_H)
		=
		\bigcup_{H'\subseteq H}
			\gamma(H').
	\end{equation}
	Let $H$ be a subgroup of $G$, and assume that every subgroup $H'$ of $G$ with $\card{H'} < \card{H}$ satisfies \Cref{eqn:rechyp-stratif-2}.
	Then:
	\begin{align*}
		V(J_H)
		& =
		\bigcup_{H'\subsetneq H}
			V(J^*_{H'})
		&
		\text{by \iref{prop:properties-Zih}{prop:properties-Zih-5}}
		\\
		& =
		\bigcup_{H'\subsetneq H}
			\bigcup_{H''\subseteq H'}
				\gamma(H'')
		&
		\text{by induction hypothesis}
		\\
		& =
		\bigcup_{H'\subsetneq H}
			\gamma(H').
	\end{align*}
	We conclude by using \Cref{prop:gamma-with-J} to relate $V(J_H)$ and $V(J^*_H) = \gamma(H) \cup V(J_H)$.
\end{proof}

\section{Nilpotent elements of the ring of components}
\label{sn:ringcomp-nearly-reduced}

In this section, we prove \Cref{thm:ringcomp-nearly-reduced}, which states that the ring of components satisfies a weak asymptotic form of reducedness.
This property is required for the proof of \Cref{thm:dim-gamma-omega}.

For the whole section, we fix a subgroup $H$ of $G$.
We let $c_H = c \cap H = \bigsqcup_{\gamma \in D} \gamma \cap H$, and we let~$D^*_H$ be the set of all conjugacy classes of $H$ contained in $c_H$.
Recall from \Cref{defn:splitting-number} that the splitting number of $H$ is $\Omega(D_H) = \card{D^*_H} - \card{D}$.
For each $n \in \N$, we denote by $R_{n,H}$ the subspace of~$R$ spanned by components of group exactly $H$ and of degree~$n$, and by $N_{n,H}$ the subspace of~$R_{n,H}$ consisting of nilpotent elements.
In other words, $R_{n,H} = (I_H \cap R^H)_n$ and $N_{n,H} = \sqrt{0} \cap R_{n,H}$.
We now state the theorem:

\begin{theorem}
	\label{thm:ringcomp-nearly-reduced}
	$\dim_k N_{n,H} = \Oo{n^{\Omega(D_H)-1}}$.
\end{theorem}

This should be compared to \cite[Proposition~3.5~(ii)]{countcomp}, which implies that
$
	\dim_k R_{n,H}
	\neq
	o\!\left(
		n^{\Omega(D_H)}
	\right)
$.
Therefore, in some sense, most high-degree elements of~$R$ are not nilpotent.
The methods used for the proof are diverse:
\begin{itemize}
	\item
		The main tool is \Cref{lem:equal-powers-implies-equal-comps-for-big-multidisc}, a result about the monoid of components.
		The proof involves the lifting invariant from \autocite{EVW2,wood}, which we review briefly in \Cref{subsn:lifting-invariant}.
	\item
		The case where $k$ has positive characteristic (coprime to $\card{G}$) is \Cref{cor:ringcomp-nearly-reduced-char-p}.
	\item
		The case of characteristic $0$ is \Cref{cor:ringcomp-nearly-reduced-char-zero}.
		We deduce it from the case of positive characteristic using classical results from model theory.
\end{itemize}

\subsection{Quick review of the lifting invariant}
\label{subsn:lifting-invariant}

In this short subsection, we recall some properties of the lifting invariant from \autocite{EVW2,wood}.

\begin{definition}
	\label{defn:multidisc}
	If $m \in \Comp_{\PC}(G , \, D , \, \xi)$ is a component of group contained in $H$, represented by a tuple $\gbar$ of elements of $G$, its \emph{$H$-multidiscriminant} is the map $\mu_H(m) : D^*_H \to \Z$ mapping a conjugacy class $\gamma \in D^*_H$ to the number of entries of $\gbar$ which belong to $\gamma$. (This is independent of $\gbar$.)
\end{definition}

\begin{definition}
	If $M$ is an integer and $m \in \Comp_{\PC}(G , \, D , \, \xi)$ is a component of group $H$, we say that $m$ is \emph{$M$-big} if its $H$-multidiscriminant $\mu_H(m)$ satisfies $\mu_H(m)(\gamma) \geq M$ for each class $\gamma \in D^*_H$.
\end{definition}

In \cite[Subsection~4.1]{countcomp}, we have defined a specific quotient $H_2(H, c_H)$ (following the notation of \cite{wood}) of the homology group~$H_2(H, \Z)$, and a morphism of monoids $\Comp_{\PC}(G , \, D , \, \xi) \to H_2(H, c_H)$ (the ``lifting invariant'', or rather its first coordinate).
The key properties of this invariant are summarized in the following theorem, which combines \cite[Theorems~4.1 and~4.5, Remark~4.3]{countcomp}:

\begin{theorem}
	\label{thm:conway-parker-red}
	There exists a constant $M$, depending only on the group $G$, such that any two $M$-big components of group $H$ with equal $H$-multidiscriminants and equal images in $H_2(H, c_H)$ are equal.
\end{theorem}

Fix an integer $M$ as in \Cref{thm:conway-parker-red}.
We conclude this subsection by recalling the lemma \cite[Lemma~4.6]{countcomp}, which implies that ``most components are $M$-big'':
\begin{lemma}
	\label{lem:few-small}
	For $n \in \N$, denote by $R_{n,H}^{M\textnormal{-small}}$ the subspace of $R$ spanned by components of group $H$ and degree $n$ which are not $M$-big.
	Then,
	$
		\dim_k
		R_{n,H}^{M\textnormal{-small}}
		=
		\Oo{n^{\Omega(D_H)-1}}
	$.
\end{lemma}

\subsection{Torsion in the monoid of components}

We use the notation from \Cref{subsn:lifting-invariant}.
In particular, we fix a constant $M$ as in \Cref{thm:conway-parker-red}.
We prove the following lemma, which implies that the monoid of components is <<~not far~>> from being torsionfree (a monoid is torsionfree if $x^r = y^r$ for $r \geq 1$ implies $x=y$):

\begin{lemma}
	\label{lem:equal-powers-implies-equal-comps-for-big-multidisc}
	Let $x, y$ be two $M$-big components of group $H$ such that, for some integer $r \in \N$ coprime with $\card{H}$, we have $x^r = y^r$.
	Then $x=y$.
\end{lemma}



\begin{proof}
	Since $x^r = y^r$, we have
	$
		\mu_H(x)
		= \frac1r \cdot \mu_H(x^r)
		= \frac1r \cdot \mu_H(y^r)
		= \mu_H(y)
	$.
	By \Cref{thm:conway-parker-red}, it remains only to show that the respective images $x', y'$ of $x$ and $y$ in $H_2(H, c_H)$ are equal.
	The equality $x^r = y^r$ implies $(x')^r = (y')^r$, which in turn implies that $x'(y')^{-1}$ is an element of $r$-torsion of the abelian group~$H_2(H, c_H)$.
	However, the group $H_2(H, c_H)$, as a quotient of the second homology group $H_2(H, \Z)$, is of $\card{H}$-torsion (cf. \cite[Chap.~VIII, \S2, Cor.~1]{serrecl}) and hence does not have any nontrivial $r$-torsion ($r$ is coprime with $\card{H}$).
	We have shown that $x'(y')^{-1} = 1$.
\end{proof}

\subsection{The case of positive characteristic}
\label{subsn:reduced-char-p}

We fix a constant $M$ as in \Cref{thm:conway-parker-red}.
Using \Cref{lem:few-small}, we also fix a constant $\tilde C$ such that, for all $n \in \N$, the number of components $m \in \Comp_{\PC}(G, D, \xi)$ of group $H$ and degree $n$ which are not $M$-big is bounded above by~$\tilde C \cdot n^{\Omega(D_H)-1}$.

\begin{corollary}
	\label{cor:ringcomp-nearly-reduced-char-p}
	Assume that the characteristic $p$ of $k$ is positive and coprime to $\card{G}$.
	Then:
	\[
		\forall n \in \N, \quad
		\dim_k N_{n,H}
		\leq
		\tilde C \cdot n^{\Omega(D_H)-1}.
	\]
\end{corollary}

\begin{proof}
	Fix some $n \in \N$.
	Let $u^r$ be the $k$-linear map $R_{n,H} \to R_{np^r,H}$ induced by $x \mapsto x^{p^r}$ on components.
	It follows from \Cref{lem:equal-powers-implies-equal-comps-for-big-multidisc} and \Cref{lem:few-small} that
	$
		\dim_k \ker(u^r)
		\leq
		\tilde C \cdot n^{\Omega(D_H)-1}
	$
	for all $r \in \N$.

	Choose a basis $x_1,\ldots,x_D$ of $N_{n,H}$, and complete it into a basis
	$
		\underbrace{
			x_1,
			\ldots,
			x_D
		}_{\in N_{n,H}},
		x_{D+1},
		\ldots,
		x_{D'}
	$
	of $R_{n,H}$.
	Note that $D= \dim N_{n,H}$ and $D'=\dim R_{n,H}$.
	Express the vectors $x_1,\ldots,x_{D'}$ in the basis $m_1,\ldots,m_{D'}$ of $R_{n,H}$ given by all components of group $H$ and of degree $n$:
	\[
		x_i
		=
		\sum_{j=1}^{D'}
			\lambda_{i,j} m_j.
	\]
	Since $x_1,\ldots,x_D$ are nilpotent, fix $r \geq 1$ such that $x_i^{p^r} = 0$ for all $i \in \{1, \ldots, D\}$.
	For $i \in \{1,\ldots, D'\}$, define:
	\[
		\tilde x_i
		\eqdef
		\sum_{j=1}^{D'}
			\lambda_{i,j}^{p^r} m_j.
	\]
	Then, $(\tilde x_1, \ldots, \tilde x_{D'})$ is still a basis of $R_{n,H}$.
	Indeed, by the properties of the Frobenius morphism, the determinant of the matrix $(\lambda_{i,j}^{p^r})_{i,j}$ is the $p^r$-th power of the determinant of the matrix $(\lambda_{i,j})_{i,j}$, which is nonzero because $x_1,\ldots,x_{D'}$ is a basis.
	In particular, the vectors $\tilde x_1,\ldots,\tilde x_D$ are linearly independent.

	Now, if $i \in \{1,\ldots,D\}$, we have:
	\[
		u^r(\tilde x_i)
		= \sum_{j=1}^{D'} \lambda_{i,j}^{p^r} u^r(m_j)
		= \sum_{j=1}^{D'} \lambda_{i,j}^{p^r} m_j^{p^r}
		=
		\left(
			\sum_{j=1}^{D'}
				\lambda_{i,j}
				m_j
		\right)^{p^r}
		= x_i^{p^r}
		= 0.
	\]
	So $\tilde x_1,\ldots,\tilde x_D$ is a list of $D$ linearly independent vectors in $\ker(u^r)$.
	The existence of such a list implies that $D \leq \dim_k \ker(u^r)$, and finally
	$
		\dim_k N_{n,H}
		=
		D
		\leq
		\dim_k \ker(u^r)
		\leq
		\tilde C \cdot n^{\Omega(D_H)-1}
	$.
\end{proof}

\subsection{The case of characteristic zero}
\label{subsn:reduced-char-zero}

The setting and the notation are the same as in \Cref{subsn:reduced-char-p}.

\begin{corollary}
	\label{cor:ringcomp-nearly-reduced-char-zero}
	Assume that the characteristic of $k$ is zero.
	Then:
	\[
		\forall n \in \N,
		\quad
		\dim_k N_{n,H}
		\leq
		\tilde C \cdot n^{\Omega(D_H)-1}.
	\]
\end{corollary}

\begin{proof}
	Since dimensions do not change under field extensions, we may assume that $k$ is algebraically closed.
	Fix an integer $n \in \N$ and denote by $F$ the finite set of components of degree $n$ and group~$H$.
	For each $r \geq 1$, define the following finite subset of $\Comp_{\PC}(G,D,\xi)$:
	\[
		F^{\wedge r}
		=
		\left\{
			m_1 m_2 \cdots m_r
			\verti
			m_1, m_2, \ldots, m_r \in F
		\right\}.
	\]

	Consider an arbitrary field $K$.
	For every $r \geq 1$ and every element
	$
		x
		=
		\sum_{m \in F}
			\lambda_m m
	$
	in the $K$-vector space spanned by~$F$, we have the following equality in $K[\Comp_{\PC}(G,D,\xi)]$:
	\begin{align*}
		x^r
		& =
		\sum_{m_1, m_2,\ldots,m_r \in F}
			\lambda_{m_1} \lambda_{m_2} \cdots \lambda_{m_r} \cdot (m_1 m_2 \cdots m_r) \\
		& =
		\sum_{m \in F^{\wedge r}}
			\left ( 
				\sum_{\substack{m_1, m_2, \ldots, m_r \in F \\ m_1 m_2 \cdots m_r = m}}
					\lambda_{m_1} \lambda_{m_2} \cdots \lambda_{m_r}
			\right )
			m.
	\end{align*}
	If $\underline{\lambda} = (\lambda_m)_{m \in F}$ is a tuple of variables indexed by $F$, we denote by $\texttt{Nilp}_r(\underline{\lambda})$ the following conjunction, which is a first-order property in the language of fields expressing the fact that the element $x=\sum_m \lambda_m m$ satisfies $x^r = 0$ in $K[\Comp_{\PC}(G,D,\xi)]$:
	\[
		\bigwedge\limits_{m \in F^{\wedge r}}
		\left(
			\sum_{\substack{
				m_1, m_2, \ldots, m_r \in F \\
				m_1 m_2 \cdots m_r = m
			}}
				\lambda_{m_1}
				\lambda_{m_2}
				\cdots
				\lambda_{m_r}
			= 0
		\right).
	\]

	Let $d = \tilde C \cdot n^{\Omega(D_H)-1} + 1$.
	If $\underline{\lambda}^{(1)}, \ldots, \underline{\lambda}^{(d)}$ are $d$ tuples of variables, each indexed by $F$, we denote by $\texttt{Indep}(\underline{\lambda}^{(1)}, \ldots, \underline{\lambda}^{(d)})$ the following first-order property in the language of fields, expressing the fact that the elements $x^{(i)} = \sum_m \lambda^{(i)}_m m$ are linearly independent:
	\[
		\forall x_1, \ldots, \forall x_d,
		\quad
		\left(
			\bigwedge_{m \in F} [
				x_1 \lambda^{(1)}_m +
				\cdots +
				x_d \lambda^{(d)}_m
				=
				0
			]
		\right)
		\implies
		\left(
			x_1 = 0 \land x_2 = 0 \land \ldots \land x_d = 0
		\right).
	\]

	Finally, we define the first-order property $\varphi_{r}$ as:
	\[
		\forall \underline\lambda^{(1)}, \ldots, \forall \underline\lambda^{(d)},
		\quad
		\left(
			\texttt{Nilp}_r(\underline\lambda^{(1)}) \land
			\ldots \land
			\texttt{Nilp}_r(\underline\lambda^{(d)})
		\right)
		\implies
		\neg \texttt{Indep}(
			\underline\lambda^{(1)},
			\ldots,
			\underline\lambda^{(d)}
		),
	\]
	expressing the fact that any elements $x^{(1)},\ldots,x^{(d)}$ of $\Span_K(F)$ whose $r$-th powers all vanish in $K[\Comp_{\PC}(G,D,\xi)]$ are linearly dependent, i.e., that the dimension of the subspace of $\Span_K(F)$ spanned by elements whose $r$-th powers are zero is at most $d-1 = \tilde C \cdot n^{\Omega(D_H)-1}$.

	By \Cref{cor:ringcomp-nearly-reduced-char-p}, the field $\overline{\mathbb{F}_p}$ satisfies the property $\varphi_{r}$ for all $r \in \N$ when $p$ is a prime not dividing~$\card{G}$.
	Let $\mathcal{U}$ be a non-principal ultrafilter on the set $\mathcal{P}$ of primes coprime to $\card{G}$, and let $\mathbb{K}$ be the ultraproduct $\prod_{p\in\mathcal{P}} \bar{\mathbb{F}_p} / \mathcal{U}$, i.e., the quotient of the ring
	$
		\prod_{p \in \mathcal{P}}
			\overline{\mathbb{F}_p}
	$
	by the (maximal) ideal generated by elements whose coordinates vanish for all primes $p$ in some set $\mathcal{P}' \in \mathcal{U}$.
	By \L o\'s's theorem \autocite[Exercise 2.5.18]{modelth}, $\mathbb{K}$ is an algebraically closed field of characteristic zero satisfying $\varphi_{r}$ for all $r \geq 1$.
	Since the theory of algebraically closed fields of characteristic zero is complete \autocite[Corollary 3.2.3]{modelth}, the same is true of~$k$.
	Therefore, the dimension of the subspace $N_{n,H,r}$ of $R_{n,H}$ spanned by elements whose $r$-th power vanishes is at most $\tilde C \cdot n^{\Omega(D_H)-1}$, for all $r \geq 1$.
	As the space $N_{n,H}$ is the increasing union of the subspaces $N_{n,H,r}$ for $r \geq 1$, we have:
	\[
		\dim_k N_{n,H}
		=
		\sup_{r \geq 1}
			\big(
				\dim_k N_{n,H,r}
			\big)
		\leq
		\tilde C \cdot n^{\Omega(D_H)-1}.
		\qedhere
	\]
\end{proof}

\section{Dimensions and degrees of the strata}
\label{sn:proof-dim-gamma-eq-zeta}

In this section, we prove \Cref{thm:dim-gamma-omega}.
We fix a nontrivial $D$-generated subgroup $H \in \Sub_{G,D}$, and we prove that the Krull dimension of the stratum $\gamma(H)$ (\Cref{defn:stratum}) is related to the splitting number~$\Omega(D_H)$ (\Cref{defn:splitting-number}) by the equality
$
	\dim_\Krull
		\gamma(H)
	=
	\Omega(D_H) + 1
$.

The section is organized in the following way:
\begin{itemize}
	\item
		In \Cref{subsn:main-lemma}, we introduce tools we need.
		We prove a bound on the Krull dimension of some quotients (\Cref{lem:dimkrull-ineq}), we recall the general form of the Hilbert-Serre theorem for algebras whose generators have unequal degrees (\Cref{thm:hilbser}), and we combine these two ingredients in \Cref{cor:dimj-hf}.
	\item
		In \Cref{subsn:proof-gamma-eq-zeta}, we compute the Krull dimension of the strata (\Cref{thm:dim-gamma-omega}) using the tools from \Cref{subsn:main-lemma} and the counting results from \cite{countcomp}.
	\item
		In \Cref{subsn:degree-gamma}, we relate the degree of $\gamma(H)$ (embedded in projective space) to the leading coefficients computed in \cite[Section~4]{countcomp}.
\end{itemize}

\subsection{Preliminaries}
\label{subsn:main-lemma}

In this subsection, we introduce tools which are later used to compute the dimension of $\gamma(H)$.
The main tool is the following lemma:

\begin{lemma}
	\label{lem:dimkrull-ineq}
	Let $B$ be a nonzero finitely generated commutative graded $k$-algebra and $J$ be a homogeneous ideal of $B$ such that, for all ideals $I$ of $B$, we have:
	\begin{equation}
		\label{eqn:hyp-on-ideal}
		I \cap J \subseteq \sqrt{0}
		\implies
		I \subseteq \sqrt{0}.
	\end{equation}
	Then, we have $\dim_\Krull(B/J) < \dim_\Krull B$.
\end{lemma}

\begin{proof}
	Let $p_1, \ldots, p_u$ be the minimal homogeneous prime ideals of $B$, whose number is finite as $B$ is Noetherian.

	Assume by contradiction that $J$ is contained in $p_i$ for some $i \in \{1,\ldots,u\}$, and let $V_i = \bigcap_{j \neq i} p_j$.
	The ideal $V_i \cap p_i$ is the intersection of all minimal homogeneous primes of~$B$, i.e., the nilradical $\sqrt{0}$.
	Since $J \subseteq p_i$, we have $V_i \cap J \subseteq V_i \cap p_i = \sqrt{0}$.
	By \Cref{eqn:hyp-on-ideal}, this implies $V_i \subseteq \sqrt{0}$.
	But then $V_i \subseteq p_i$ with $p_i$ prime, and thus some $p_j$ with $j \neq i$ must be contained in $p_i$, contradicting the minimality of~$p_i$.
	We have shown that $J$ is not contained in any $p_i$.

	Consider a maximal chain
	$
		J \subseteq
		q_0 \subsetneq q_1 \subsetneq \ldots \subsetneq q_{\dim_\Krull(B/J)}
	$
	of prime ideals of~$B$ containing $J$.
	By the previous paragraph, the inclusion $J \subseteq q_0$ implies that $q_0$ is not a minimal homogeneous prime ideal of~$B$.
	Therefore, $q_0$ contains some $p_i$ for $i \in \{1, \ldots, u\}$.
	We then have the chain
	$
		p_i \subsetneq q_0 \subsetneq q_1 \subsetneq \ldots \subsetneq q_{\dim_\Krull(B/J)}
	$
	of homogeneous prime ideals of $B$, proving that $\dim_\Krull B > \dim_\Krull(B/J)$.
\end{proof}

A second tool is \Cref{thm:hilbser}, which is a variant of the Hilbert-Serre theorem for graded algebras whose generators are not necessarily of degree $1$.
The case where all generators have degree $1$ is classical (see for example \autocite[Theorem I.7.5]{hartshorne}), and in that case the Hilbert function eventually coincides with a polynomial.
In the general case, there is instead a finite list of polynomials through which the Hilbert function eventually cycles periodically: we have a \emph{Hilbert quasi-polynomial}.

\begin{theorem}[Hilbert-Serre]
	\label{thm:hilbser}
	Let $A$ be a commutative graded $k$-algebra of positive Krull dimension, generated by finitely many generators of respective degrees $d_1, \ldots, d_N$.
	Let $W = \lcm(d_1, \ldots, d_N)$.
	Then, there exist (uniquely defined) polynomials $Q_0, \ldots, Q_{W-1}$ such that:
	\begin{enumerate}[label=(\roman*)]
		\item
			\label{thm:hilbser-1}
			For all $m \in \{0, \ldots, W-1\}$ and $n$ large enough, we have
			$
				\HF_A(W n + m) = Q_m(n)
			$;
		\item
			\label{thm:hilbser-2}
			$\deg Q_0 = \dim_\Krull(A) - 1$;
		\item
			\label{thm:hilbser-3}
			$\deg Q_m \leq \dim_\Krull(A) - 1$ for all $m \in \{1, \ldots, W-1\}$.
	\end{enumerate}
\end{theorem}

For proofs, we refer the reader to \autocite[Theorem~4.4.3]{cohmac} or \cite[Lemmas~5.4.4, 5.4.5 and 5.4.6]{SegThese}.

To say that two functions $f,g$ satisfy $f=O(g)$ and $f \neq o(g)$, we write $f = \Os{g}$.
With this notation, \Cref{thm:hilbser} implies the following, for any algebra $A$ satisfying the hypotheses:
\begin{equation}
	\label{eqn:hilbfun-osharp}
	\HF_A(n) = \Os{n^{\dim_\Krull(A)-1}}.
\end{equation}
Together, \Cref{thm:hilbser} and \Cref{lem:dimkrull-ineq} directly imply:

\begin{corollary}
	\label{cor:dimj-hf}
	Under the hypotheses of \Cref{lem:dimkrull-ineq}, we have
	$
		\dim_k J_n =
		\HF_B(n) +
		\Oo{n^{\dim_\Krull(B)-2}}
	$.
\end{corollary}

\subsection{The dimension of $\gamma(H)$}
	\label{subsn:proof-gamma-eq-zeta}

We are ready to prove \Cref{thm:dim-gamma-omega}~: the Krull dimension of the stratum $\gamma(H)$ is one more than the splitting number~$\Omega(D_H)$.

\begin{proof}[Proof of \Cref{thm:dim-gamma-omega}]
	Let $R_{n,H} = (I_H \cap R^H)_n$ be the space spanned by components of group $H$ and of degree $n$.
	In order to apply \Cref{cor:dimj-hf}, we check that the hypotheses of \Cref{lem:dimkrull-ineq} hold when
	$
		B
		= R^H/(\Gamma_H \cap R^H)
	$
	and $J$ is the image of $I_H \cap R^H$ in $B$.
	First, $B$ is a finitely generated commutative graded $k$-algebra (cf. \Cref{subsn:ringcomp}) and $J$ is a homogeneous ideal of $B$.
	Note also that the ideal $\Gamma_H$ being radical implies that $B$ is reduced and hence $\sqrt{0} = 0$ in $B$.
	We consider an ideal $I$ of~$B$ such that $I \cap J = 0 $, and we want to show $I = 0$.
	Let $\tilde I$ be the inverse image of $I$ in $R^H$.
	Then, $I \cap J = 0$ rewrites as $\tilde I \cap I_H \subseteq \Gamma_H \cap R^H$.
	By \iref{prop:GammaH}{prop:GammaH-vi}, this implies $\tilde I \subseteq \Gamma_H$, i.e., $I = 0$.
	We have checked the hypotheses of \Cref{lem:dimkrull-ineq}.
	We obtain:
	\begin{align*}
		\dim_k\,J_n
		& =
		\HF_{R^H/(\Gamma_H \cap R^H)}(n)
		+
		\Oo{n^{\dim_\Krull(R^H/(\Gamma_H \cap R^H))-2}}
		&
		\text{by \Cref{cor:dimj-hf}}
		\\
		& =
		\Os{n^{\dim_\Krull(R^H/(\Gamma_H \cap R^H))-1}}
		&
		\text{by \Cref{eqn:hilbfun-osharp}}
		\\
		& =
		\Os{n^{\dim_\Krull \gamma(H) - 1}}
		&
		\text{by}
		\begin{cases}
			R^H/(\Gamma_H \cap R^H) \simeq R/\Gamma_H \\
			\text{\iref{prop:GammaH}{prop:GammaH-iii}}
		\end{cases}
	\end{align*}
	By \iref{prop:GammaH}{prop:GammaH-iv}, the intersection $\Gamma_H \cap I_H$ is contained in $\sqrt{I^*_H}$.
	As one cannot obtain nonzero terms involving components of group strictly larger than $H$ by taking powers of linear combinations of components of group exactly $H$, elements of $\sqrt{I^*_H} \cap I_H \cap R^H$ are exactly the nilpotent elements of $I_H \cap R^H$.
	Thus $(\Gamma_H \cap I_H \cap R^H)_n \subseteq (\sqrt{I^*_H} \cap I_H \cap R^H)_n = (\sqrt{0} \cap I_H \cap R^H)_n = N_{n,H}$.
	We have:
	\begin{align*}
		\dim_k\,J_n
		& =
		\dim_k (I_H \cap R^H)_n
		-
		\dim_k (\Gamma_H \cap I_H \cap R^H)_n
		&
		\text{by definition of $J$}
		\\
		& =
		\dim_k R_{n,H}
		+
		\Oo{\dim_k N_{n,H}}
		&
		\text{as }
		\left\{
			\begin{array}{ccc}
				(I_H \cap R^H)_n &
				= &
				R_{n,H}
				\\
				(\Gamma_H \cap I_H \cap R^H)_n &
				\subseteq &
				N_{n,H}
			\end{array}
		\right.
		\\
		& =
		\Os{n^{\Omega(D_H)}}
		+
		\Oo{n^{\Omega(D_H)-1}}
		&
		\text{by }
		\begin{cases}
			\text{\cite[Theorem~3.2~(ii)]{countcomp}} \\
			\text{\Cref{thm:ringcomp-nearly-reduced}}
		\end{cases}
		\\
		& =
		\Os{n^{\Omega(D_H)}}.
	\end{align*}
	Comparing the two estimates for $\dim_k \, J_n$, we obtain $\dim_\Krull \gamma(H) - 1 = \Omega(D_H)$.
\end{proof}

\subsection{The degree of $\gamma(H)$}
\label{subsn:degree-gamma}

In this subsection, we use the results of \cite[Section~4]{countcomp} to obtain more precise results concerning the strata $\gamma(H)$.
We come back to where the proof of \Cref{thm:dim-gamma-omega} ended, including all notation and hypotheses.
Since we now know that $\dim_\Krull \gamma(H) - 1 = \Omega(D_H)$, we obtain (again by expressing $\dim_k \, J_n$ in two different ways):
\begin{equation}
	\label{eqn:dimrnh-compared-to-hf}
	\dim_k R_{n,H}
	=
	\HF_{R^H/(\Gamma_H \cap R^H)}(n) + \Oo{n^{\Omega(D_H)-1}}.
\end{equation}

Assume that all non-factorizable elements of $\Comp_{\PC}(G,D,\xi)$ have the same degree $W$, and let~$N$ be their number.
Then, $\Proj\,R$ embeds in the projective space $\Pr^{N-1}(k)$ and we can compute the degree%
\footnote{
	A similar computation is possible when non-factorizable elements have different degrees -- one should then average the coefficients in front of $n^{\Omega(D_H)}$ of the polynomials $Q_0, \ldots, Q_{W-1}$ of \Cref{thm:hilbser} --, but there does not seem to be a standard notion of degree for subspaces of weighted projective spaces.
} of the closed subset $\bar{\gamma(H)}$ (the degree of a non-irreducible subset is the sum of the degrees of its irreducible components of maximal Krull dimension).
Note that we consider $\Proj$ instead of $\Spec$: we have quotiented out by the action of $k^{\times}$ and all dimensions are one less than previously computed.
In this case, we know that $\HF_{R^H/(\Gamma_H \cap R^H)}(n)$ and $\dim_k R_{n,H}$ vanish whenever~$n$ is not a multiple of~$W$, and coincide with two polynomials of degree $\Omega(D_H)$ when evaluated at large enough multiples of~$W$.
Moreover, \Cref{eqn:dimrnh-compared-to-hf} implies that the leading monomials of these two polynomials are equal:
\[
	\dim_k(R_{W n,H}) \sim \HF_{R^H/(\Gamma_H \cap R^H)}(W n) \sim \alpha \cdot (W n)^{\Omega(D_H)}
	\quad
	\text{for some } \alpha > 0.
\]
Classically, the number $\Omega(D_H)! \cdot \alpha$ is the degree of the closed subset $\bar{\gamma(H)}$ of $\Pr^{N-1}(k)$.
Recall that the group $H_2(H, c_H)$ (mentioned in \Cref{subsn:lifting-invariant}) is defined in \cite{wood} or \cite[Definition~4.4]{countcomp}.
The results of \cite{countcomp} give the value of $\alpha$, and consequently of the degree of $\bar{\gamma(H)}$, in various situations:
\begin{itemize}
	\item
		If $H$ is a non-splitter, then $\Omega(D_H)=0$ by definition, and $\alpha = \card{H_2(H,c_H)}$ by \cite[Proposition~4.11]{countcomp}.
		In this case, $\bar{\gamma(H)}$ (embedded in the projective space $\Pr^{N-1}(k)$) is of dimension zero: it is a union of finitely many points.
		The degree is precisely the number of these points: there are $\card{H_2(H,c_H)}$ of them.
	
		\begin{remark}
			\label{rk:nb-lines-gamma-nonsplitt}
			If one sees $\bar{\gamma(H)}$ as embedded in affine space $\mathbb{A}^N(k)$ (as we have done in other sections), it is a union of $\card{H_2(H,c_H)}$ lines going through the origin $0$.
		\end{remark}
	\item
		Assume that $D$ consists of a single conjugacy class $c$ of $G$ and $\xi(c) = 1$.
		Let $\ord(c)$ be the order of any element contained in the conjugacy class $c$.
		Let $s$ be the number of conjugacy classes of~$H$ contained in $c \cap H$, so that $\Omega(D_H)=s-1$.
		By \cite[Corollary~4.14]{countcomp} and the comments underneath it, we have:
		\[
			\alpha = \frac{\ord(c)\card{H_2(H, c_H)}}{\card{H^{\ab}}(s-1)!}.
		\]
		The degree of $\bar{\gamma(H)}$ is then given by:
		\[
			\frac{
				\ord(c)
				\cdot
				\card{H_2(H, c_H)}
			}{
				\card{H^{\ab}}
			}.
		\]
\end{itemize}

\section{Explicit description of the variety of components}
\label{sn:spectrum-2}

This section is organized as follows:
\begin{itemize}
	\item
		In \Cref{subsn:line-in-gamma}, we describe a curve which is contained in each stratum $\gamma(H)$ (\Cref{thm:curve-in-gamma}).
		In some situations, we prove that this curve is all of $\gamma(H)$.
	\item
		In \Cref{subsn:free-factor}, we define the notion of a free factor family of subgroups (\Cref{defn:free-family-subgp} and \Cref{defn:factor-family}).
		In \Cref{prop:carac-factor-free-tensor}, we state key properties of this notion.
	\item
		In \Cref{subsn:factored-ringcomp}, we describe the stratum corresponding to the product of a free factor family in terms of the strata associated to its factors (\Cref{prop:desc-stratum-product}).
		We then use this to obtain a complete description of the variety of components under specific hypotheses (\Cref{thm:desc-spectrum}).
	\item
		In \Cref{ssn:desc-varcomp-coord}, we rephrase the description explicitly in terms of coordinates (\Cref{thm:desc-spectrum-coords}).
	\item
		In \Cref{ssn:symgp}, we apply \Cref{thm:desc-spectrum-coords} to obtain a complete description of the variety of components in the classical case of symmetric groups, using the results of \cite[Chapter~6]{SegThese}.
\end{itemize}

\subsection{A curve in each stratum}
\label{subsn:line-in-gamma}

Fix a nontrivial $D$-generated subgroup $H$ of $G$.
In this subsection, we describe a curve which is contained in the stratum $\gamma(H)$, and give a criterion for it to be all of $\gamma(H)$.
We start by defining ideals $A^\circ_H$ and $A_H$:

\begin{definition}
	We let $A^\circ_H$ be the homogeneous ideal of $R$ generated by the differences $m-m'$, for pairs $(m,\,m')$ of components of same degree whose group is contained in $H$.
	We also define $A_H \eqdef A^\circ_H + J^*_H$.
\end{definition}

We let $p_1, \ldots, p_M$ be the non-factorizable components whose group is contained in $H$.
In the quotient $R/A_H$, components whose group is not contained in $H$ are made to vanish, and other components are made to be equal whenever their degrees match.
In other words, instead of being the monoid algebra of the monoid of components $\Comp_{\PC}(G , \, D , \, \xi)$, the quotient $R/A_H$ turns the ring of components into the monoid algebra of the submonoid of $\Zpos$ formed of integers which are the degree of some component whose group is contained in $H$.
A more concrete description of the ideal~$A_H$, of the quotient $R/A_H$, and of the closed subset $V(A_H) \subseteq \Spec\,R$ is given by the following theorem:

\begin{theorem}
	\label{thm:curve-in-gamma}
	let $p_1, \ldots, p_M$ be the non-factorizable components whose group is contained in $H$, let~$d(1), \ldots, d(M)$ be their respective degrees, and let $D$ be the greatest common divisor of the integers~$d(i)$.
	Then:
	\begin{enumerate}[label=(\roman*)]
		\item
			\label{thm:curve-in-gamma-i}
			The quotient $R/A_H$ is isomorphic to the graded subalgebra
			$
				k[X^{d(1)}, \ldots, X^{d(M)}]
			$
			of $k[X]$.
		\item
			$A_H$ is a homogeneous prime ideal of $R$, and the Krull dimension of $R/A_H$ is $1$.
			Hence, $V(A_H)$ is a $1$-dimensional irreducible closed subset of $\Spec\,R$, contained in $V(J^*_H)$.
		\item
			We have $\sqrt{A_H + I_H} = \varpi$ and $V(A_H) \setminus \{\varpi\} \subseteq \gamma(H)$.
		\item
			\label{thm:curve-in-gamma-iv}
			If there is at most one component of group $H$ for each degree,%
			\footnote{
				This hypothesis implies that $H$ is a non-splitter and $H_2(H,c_H) = 1$ by \cite[Theorem~1.4]{countcomp}.
				This result is an <<~effective~>> version of the case $H_2(H,c_H) = 1$ of \Cref{rk:nb-lines-gamma-nonsplitt}.
			}
			then $V(A_H) = \gamma(H) \cup \{\varpi\}$.
		\item
			\label{thm:curve-in-gamma-v}
			Assume that $k$ is algebraically closed and let $p_{M+1}, \ldots, p_N$ be the non-factorizable components whose group is not contained in $H$.
			Embed the set of $k$-points of $\Spec\,R$ into $k^N$ as in \Cref{subsubsn:affine-embed}.
			Then, the $k$-points of $V(A_H)$ form a ``weighted'' line, consisting of points of the following form, for some $\lambda \in k$:
			\[
				(
					\underbrace{
						\lambda^{d(1)/D}, \lambda^{d(2)/D}, \ldots, \lambda^{d(M)/D}
					}_{\substack{
						\text{non-factorizable components}\\
						\text{of group } \subseteq H
					}}
					\,\, , \,\,
					\underbrace{
						0, \ldots, 0
					}_{
						\text{other non-factorizable components}
					}
				).
			\]
			In particular, under the hypothesis of \ref{thm:curve-in-gamma-iv}, the $k$-points of $\gamma(H)$ are points of the form above for $\lambda \in k^{\times}$.
	\end{enumerate}
\end{theorem}

\begin{proof}
	~
	\begin{enumerate}[label=(\roman*)]
		\item
			The image of a component $m$ (whose group is contained in $H$) in the quotient $R/A_H$ is entirely determined by the degree of $m$, so that we can denote its image by $X^{\deg(m)}$.
			Applying this to each non-factorizable component whose group is contained in~$H$ shows that $R/A_H \simeq R^H/(A^\circ_H \cap R^H)$ is a quotient of $k[X^{d(1)}, \ldots, X^{d(M)}]$.
			Relations between the non-factorizable components are generated by equalities of the form $m_1 \cdots m_u = m'_1 \cdots m'_u$ for non-factorizable components $m_1, \ldots, m_u, m'_1, \ldots, m'_u$ satisfying $\deg(m_1) + \cdots + \deg(m_u) = m'_1 + \cdots + \deg(m'_u)$.
			But such relations are automatically satisfied in $k[X^{d(1)}, \ldots, X^{d(M)}]$, so that $R/A_H \simeq k[X^{d(1)}, \ldots, X^{d(M)}]$.
		\item
			The quotient $R/A_H \simeq k[X^{d(1)}, \ldots, X^{d(M)}]$ is a graded subalgebra of the domain $k[X]$ and hence is integral.
			This proves that $A_H$ is a prime ideal of $R$.
			Moreover, the Krull dimension of $R/A_H$ equals the transcendence degree of its field of fractions $\Frac(R/A_H) \simeq \Frac(k[X^{d(1)}, \ldots, X^{d(M)}]) = k(X^D)$, which is $1$.
			The claims concerning $V(A_H)$ follow.
		\item
			Since $H$ is $D$-generated, the product $p_1 \cdots p_M$ has group exactly $H$ and hence belongs to $I_H$.
			The image of that element in $R/A_H$ is $X^{d(1)+\cdots+d(M)}$, and thus $R/(A_H + I_H)$ is a quotient of the algebra $k[X^{d(1)}, X^{d(2)}, \ldots, X^{d(M)}]/(X^{d(1)+\cdots+d(M)})$, in which all homogeneous elements of positive degree are nilpotent.
			Therefore, the reduced quotient of $R/(A_H + I_H)$ is~$k$, which implies $\sqrt{A_H + I_H} = \varpi$.
			It follows that $V(A_H) \cap V(I_H) = \{\varpi\}$, i.e., that $V(A_H) \setminus \{\varpi\}$ is contained in $V(J^*_H) \setminus V(I_H) = \gamma(H)$.
		\item
			We have shown that $V(A_H) \subseteq \gamma(H) \cup \{\varpi\}$, hence it remains only to show that $V(A_H) \supseteq \gamma(H)$, i.e., that $V(A_H) \cup V(I_H) = V(I^*_H)$.
			It suffices to show that $A_H \cap I_H = I^*_H$.
			Let $S$ be the submonoid of $\Zpos$ formed of integers which are the degree of some component of group $H$, and let $F_k$ be the only component of group $H$ of degree $k$ for $k \in S$.
			An element of $I_H$ is of the form:
			\[
				x
				=
				\sum_{k \in S}
					\lambda_k F_k
				+
				\underbrace{
					x_{\supsetneq H}
				}_{\in I^*_H}.
			\]
			The image of~$x$ in $R/A_H$ (seen as a subring of $k[X]$ via point \ref{thm:curve-in-gamma-i}) is $\sum_{k \in S} \lambda_k X^k$.
			So, $x$ belongs to~$A_H$ exactly when the coefficients $\lambda_k$  vanish for all $k \in S$, if and only if $x = x_{\supsetneq H}$ belongs to~$I^*_H$.
		\item
			A point $(x_1,\ldots,x_N) \in (\Spec\,R)(k)$ is in $V(A_H)(k)$ if $x_{M+1} = \ldots = x_N = 0$ (because~$A_H$ contains~$J^*_H$), and if two products of the coordinates $x_i$ are equal whenever the sums of the degrees of the corresponding non-factorizable components $p_i$ match.
			Using Bézout's identity, fix integers $a_1, \ldots, a_M \in \Z$ such that:
			\[
				D
				=
				\sum_{i=1}^M
					a_i d(i).
			\]
			For each $i \in \{1,\ldots,M\}$, we have in $R/A_H$:
			\[
				p_i
				\left(
					\prod_{j \text{ s.t. } a_j \leq 0}
						p_j^{-a_j}
				\right)^{d(i)/D}
				=
				\left(
					\prod_{j  \text{ s.t. } a_j \geq 0}
						p_j^{a_j}
				\right)^{d(i)/D}
			\]
			because both sides are components of same degree, since $d(i) = \frac{d(i)}{D} \cdot \sum_j a_j d(j)$.
			Hence, $(x_1,\ldots,x_M,0,\ldots,0) \in V(A_H)(k)$ means that we have, for all $i \in \{1,\ldots,M\}$:
			\[
				x_i =
				\left(
					\prod_{j} x_j^{a_j}
				\right)^{d(i)/D}
				=
				\lambda^{d(i)/D}
			\]
			where $\lambda = \prod_j x_j^{a_j}$ does not depend on $i$.
			So, points of $V(A_H)(k)$ are of the form:
			\[
				(\lambda^{d(1)/D}, \lambda^{d(2)/D},\ldots,\lambda^{d(M)/D},0,\ldots,0)
				\text{ with } \lambda \in k.
			\]
			The coordinates of such a point satisfy the equalities defining $R$, since these are equalities between components of same degree and same group.
			\qedhere
	\end{enumerate}
\end{proof}

\begin{remark}
	The principle of \Cref{thm:curve-in-gamma} is to forget everything about a given component besides its degree and the fact that its group is contained in $H$.
	Finer descriptions of $\gamma(H)$ may be obtained by forgetting less information about these components, for example by replacing $A_H$ by the ideal generated by the differences $m-m'$ between two components of same group $H$, same $H$-multidiscriminant (\Cref{defn:multidisc}) and same image in $H_2(H, c_H)$ (cf. \Cref{subsn:lifting-invariant}), i.e., with the same lifting invariant in the sense of \cite{wood}.
\end{remark}

\subsection{Free families and factor families}
\label{subsn:free-factor}

\begin{definition}
	\label{defn:free-family-subgp}
	A list $H_1,\ldots,H_k$ of subgroups of $G$ is a \emph{free family} if the two following conditions are met:
	\begin{itemize}
		\item
			For all disjoint subsets $A, B$ of $\{1,\ldots,k\}$, we have $\gen{(H_i)_{i \in A}} \cap \gen{(H_j)_{j \in B}} = 1$.
		\item
			Elements of $H_i$ commute with elements of $H_j$ when $i \neq j$.
	\end{itemize}
\end{definition}

\begin{definition}
	\label{defn:prod-free-family-subgp}
	If $H_1, \ldots, H_k$ is a free family of subgroups of $G$, the (Zappa-Szép) \emph{product} of this family is the subgroup $\gen{H_1, \ldots, H_k}$ of $G$.
\end{definition}

As an abstract group, the product of a free family $H_1 ,\ldots, H_k$ is isomorphic to the direct product of the groups $H_i$.
We now fix a subgroup $H \in \Sub_{G,D}$.

\begin{definition}
	\label{defn:factor-family}
	A \emph{factor family} of $H$ is a list $H_1,\ldots,H_k$ of $D$-generated subgroups of~$H$ such that every non-factorizable component whose group is contained in $H$ has its group contained in some $H_i$.
\end{definition}

\begin{proposition}
	\label{prop:generated-by-factor-family}
	Let $H_1,\ldots,H_k$ be a factor family of $H$.
	Then $H_1,\ldots,H_k$ generate~$H$.
\end{proposition}

\begin{proof}
	Use \Cref{prop:dgen-are-monodromy-groups} to fix a component $m$ of group $H$, and write $m$ as a product of non-factorizable components $m = m_1 \cdots m_r$.
	Since $H_1,\ldots,H_k$ is a factor family of $H$, each factor $m_i$ has its group contained in some $H_j$.
	Hence $H = \gen{m} = \gen{\gen{m_1},\ldots,\gen{m_r}}$ is contained in $\gen{H_1,\ldots,H_k}$.
\end{proof}

\begin{proposition}
	\label{prop:carac-factor-free-tensor}
	Let $H_1,\ldots,H_k$ be a list of subgroups of $H$.
	Denote by $\Phi$ the following morphism:
	\[
		\Phi :
		\left\lbrace
		\begin{matrix}
			R^{H_1} \otimes \cdots \otimes R^{H_k} &
			\to &
			R^H
			\\
			m_1 \otimes \cdots \otimes m_k &
			\mapsto &
			m_1 \cdots m_k
		\end{matrix}
		\right.
		.
	\]

	Then:
	\begin{enumerate}[label=(\roman*)]
		\item
			\label{prop:carac-factor-free-tensor-i}
			$H_1,\ldots,H_k$ is a factor family if and only if $\Phi$ is surjective.
		\item
			\label{prop:carac-factor-free-tensor-ii}
			If $H_1,\ldots,H_k$ is free, then $\Phi$ is injective.
		\item
			\label{prop:carac-factor-free-tensor-iii}
			Assume that $H_1, \ldots, H_k$ is a free factor family of $H$.
			Then:
			\[
				\Phi^{-1}(I_H \cap R^H) = (I_{H_1} \cap R^{H_1}) \otimes \cdots \otimes (I_{H_k} \cap R^{H_k}).
			\]
	\end{enumerate}
\end{proposition}

\begin{proof}
	~
	\begin{enumerate}[label=(\roman*)]
	\item Let us prove point (i).
		\begin{itemize}
			\item[$(\Leftarrow)$]
				Assume that $\Phi$ is surjective and consider a non-factorizable component $m \in R^H$.
				Since~$m$ is in the image of $\Phi$ and is a component (not a sum thereof), it equals a product $m_1 \cdots m_k$ with $m_i$ a component in $R^{H_i}$.
				Since $m$ is non-factorizable, we have $m=m_i$ for some $i$ and thus $\gen{m} \subseteq H_i$.

			\item[$(\Rightarrow)$]
				Conversely, assume $H_1,\ldots,H_k$ is a factor family.
				To prove that $\Phi$ is surjective, it is enough to show that its image contains any component $m$ whose group is contained in~$H$.
				Decompose $m$ as a product $m_1 \cdots m_r$ of non-factorizable components.
				By hypothesis, each factor~$m_j$ is contained in some $H_{\psi(j)}$.
				For $i \in \{1,\ldots,k\}$, let $m'_i$ be the product of the components~$m_j$ such that $\psi(j) = i$.
				Then $m = \Phi(m'_1\otimes\cdots\otimes m'_k)$.
		\end{itemize}

	\item
		Assume that the family $(H_i)$ is free.
		Since $\Phi$ maps pure tensors of components to components, and since components form a basis, it suffices to check that $\Phi$ maps distinct pure tensors of components to distinct components.
		Thus, we assume that there are two tuples $(m_1, \ldots, m_k), (m'_1, \ldots, m'_k)$ of components, with $\gen{m_i}$ and $\gen{m'_i}$ contained in $H_i$, such that:
		\[
			\prod_{j=1}^k
				m_j
			=
			\prod_{j=1}^k
				m'_j
			.
		\]
		Consider a braid relating two tuples representing these two components, and decompose it into elementary braids.
		Since the subgroups $H_i$ commute with each other, elementary braids swapping two elements from different subgroups $H_i$ do not change these elements, and do not affect the position of these elements among elements of the same subgroup.
		Hence, fixing some $i \in \{1,\ldots,k\}$ and retaining only elementary braids which do exchange elements of the same subgroup $H_i$ yields a braid relating $m_i$ and $m'_i$, proving that $m_i = m'_i$.
		Therefore, $\Phi$ is injective.
	\item
		First note that $\Phi$ is an isomorphism by the previous points.
		That $\Phi((I_{H_1} \cap R^{H_1}) \otimes \cdots \otimes (I_{H_k} \cap R^{H_k})) \subseteq I_H \cap R^H$ follows immediately from the fact that a product of components $m_1 \cdots m_k$ with $\gen{m_i} = H_i$ has group $\gen{H_1, \ldots, H_k} = H$ (cf. \Cref{prop:generated-by-factor-family}).

		Conversely, consider nonzero constants $\lambda_j \in k$ and components $m_{i,j}$ with $\gen{m_{i,j}} \subseteq H_i$ (such that no two tuples $(m_{i,j})_{1 \leq i \leq k}$ are equal for distinct values of $j$), such that:
		\[
			\Phi\!\left(
				\sum_j
					\lambda_j \cdot
					(m_{1,j} \otimes \cdots \otimes m_{k,j})
			\right)
			\in
			I_H \cap R^H.
		\]
		Then, each component $m_{1,j} \cdots m_{k,j}$ has group exactly $H$.
		Fix a $j$.
		The groups~$\gen{m_{i,j}}$ for $i \in \{1, \ldots, k\}$, being subgroups of $H_i$, form a free family of subgroups of $H$ and hence the group generated by $m_{1,j} \cdots m_{k,j}$ is isomorphic to the direct product $\prod_{i=1}^k \gen{m_{i,j}}$.
		We have $\card{H} = \prod_{i=1}^k \card{H_i}$ but also $\card{H} = \card{\gen{m_{1,j} \cdots m_{k,j}}} = \prod_{i=1}^k \card{\gen{m_{i,j}}}$ with $\card{\gen{m_{i,j}}} \leq \card{H_i}$, which implies that the component $m_{i,j}$ has group exactly $H_i$, for each $i \in \{1,\ldots,k\}$.
		Therefore, the element
		$
			\sum_j
				\lambda_j \cdot
				(m_{1,j} \otimes \cdots \otimes m_{k,j})
		$
		belongs to $(I_{H_1} \cap R^{H_1}) \otimes \cdots \otimes (I_{H_k} \cap R^{H_k})$, which proves the claim.
		\qedhere
	\end{enumerate}
\end{proof}

Finally, we give a group-theoretic criterion to identify some factor families:

\begin{proposition}
	\label{prop:criterion-factor-family}
	Assume that $\card{D}=1$ and $\xi=1$, i.e., we consider a single conjugacy class $c$.
	Assume that~$H$ is a subgroup of $G$ generated by subgroups $H_1,\ldots,H_k$ such that:
	\begin{enumerate}[label=(\roman*)]
		\item
			\label{prop:criterion-factor-family-hyp1}
            For all $i \in \{1,\ldots,k\}$, the subgroup $H_i$ has a trivial intersection with $\gen{(H_j)_{j \neq i}}$.
		\item
			\label{prop:criterion-factor-family-hyp2}
            $
                \displaystyle
                c \cap H
                =
                \bigsqcup_{i=1}^k
                    (c \cap H_i)
            $.
	\end{enumerate}

	Then $H_1,\ldots,H_k$ is a factor family of $H$.
\end{proposition}

\begin{proof}
	For $i \in \{1,\ldots,k\}$, let $c_i = c \cap H_i$.
	Consider a non-factorizable component $m$ whose group is contained in $H$.
	By \ref{prop:criterion-factor-family-hyp2}, and using the fact that braids can permute conjugacy classes freely, we may choose a tuple $\gbar$ representing the component $m$ of the form:
    \[
        \gbar
		=
		(g_{1,1},\ldots,g_{1,n(1)},\ldots,g_{k,1},\ldots,g_{k,n(k)})
		\qquad
		\text{with }
		g_{i,j} \in c_i.
    \]
	Let $\pi_i = g_{i,1} \cdots g_{i,n(i)} \in H_i$.
	We have $\pi_1 \cdots \pi_k = 1$ and thus $\pi_i = \left ( \pi_{i+1} \cdots \pi_k \pi_1 \cdots \pi_{i-1} \right ) ^{-1} \in \gen{(H_j)_{j\neq i}}$.
	By \ref{prop:criterion-factor-family-hyp1}, this implies $\pi_i=1$.
	Therefore, all the tuples $\gbar_i = (g_{i,1},\ldots,g_{i,n(i)}) \in c_i^{n(i)}$ define components.
	Since~$m$ is non-factorizable, it equals one of its factors and thus $\gen{m}$ is contained in some~$H_i$.
\end{proof}

\subsection{The stratum associated to the product of a free factor family}
\label{subsn:factored-ringcomp}

We fix a subgroup $H \in \Sub_{G,D}$.

\subsubsection{The set $\tilde\gamma(H)$.}

\begin{definition}
	We denote by $\tilde\gamma(H)$ the open subset of $\Spec\,R^H$ obtained as the complement of $V(I_H \cap R^H)$.
	In other words, $\tilde\gamma(H)$ is the set of all prime ideals of $\Spec \, R^H$ for which some component~$m$ of group exactly $H$ does not belong to $p$.
\end{definition}

We denote by $\pi_H$ the dominant map $\Spec\,R \to \Spec\,R^H$ induced by the inclusion $R^H \hookrightarrow R$, and by $\iota_H$ the embedding $\Spec\,R^H \to \Spec\,R$ induced by the surjection $R \twoheadrightarrow R/J^*_H \simeq R^H$.
Since the composition $R^H \hookrightarrow R \twoheadrightarrow R/J^*_H \simeq R^H$ is the identity map, we have the equality:
\begin{equation}
	\label{eqn:iota-then-pi}
	\pi_H \circ \iota_H = \id_{\Spec \, R^H}
\end{equation}

The sets $\gamma(H)$ and $\tilde\gamma(H)$ determine each other, as shown below:

\begin{proposition}
	\label{prop:gamma-and-tilde-gamma}
	The sets $\gamma(H)$ and $\tilde\gamma(H)$ are related by the following equalities:
	\begin{enumerate}[label=(\roman*)]
		\item
			\label{prop:gamma-and-tilde-gamma-i}
			$
				\gamma(H) =
				\iota_H(\tilde\gamma(H))
			$
		\item
			\label{prop:gamma-and-tilde-gamma-ii}
			$
				\tilde\gamma(H) = \pi_H(\gamma(H))
			$
		\item
			\label{prop:gamma-and-tilde-gamma-iii}
			$
				\gamma(H)
				=
				V(J^*_H)
				\cap
				\pi_H^{-1}(\tilde\gamma(H))
			$
	\end{enumerate}
\end{proposition}

\begin{proof}
	~
	\begin{enumerate}[label=(\roman*)]
		\item
			\begin{itemize}
				\item[$(\subseteq)$]
					Let $p \in \gamma(H)$.
					Then, $p$ is a prime ideal of $R$ containing $J^*_H$, and not containing some component~$m$ of group $H$.
					Since $p$ contains $J^*_H$, we have $p = J^*_H + (p \cap R^H) = \iota_H(p \cap R^H)$.
					Moreover, $p \cap R^H$ is a prime ideal of $R^H$ not containing $m$, and thus $p \cap R^H \in \tilde\gamma(H)$.
				\item[$(\supseteq)$]
					Let $p \in \tilde\gamma(H)$.
					Then, $p$ is a prime ideal of $R^H$ not containing some component $m$ of group~$H$.
					The prime ideal $\iota_H(p) = p + J^*_H$ of $R$ contains~$J^*_H$ but does not contain $m$, and thus it belongs to $\gamma(H)$.
			\end{itemize}
		\item
			Follows from \ref{prop:gamma-and-tilde-gamma-i} by applying $\pi_H$ to both sides, and using \Cref{eqn:iota-then-pi}.
		\item
			The inclusion $\subseteq$ follows from \Cref{prop:gamma-with-J} and from \ref{prop:gamma-and-tilde-gamma-ii}.
			Conversely, let $p \in 
			V(J^*_H)
			\cap
			\pi_H^{-1}(\tilde\gamma(H))$.
			Since $p$ contains $J^*_H$, we have $p = J^*_H + (p \cap R^H)$, i.e. $p = J^*_H + \pi_H(p)$.
			By hypothesis, $\pi_H(p) \in \tilde\gamma(H)$ and hence there is some component $m$ of group $H$ which is not contained in $\pi_H(p) = p \cap R^H$.
			But then $m$ is also not contained in $p$ and thus $p \in \gamma(H)$.
			\qedhere
	\end{enumerate}
\end{proof}

\begin{remark}
	We can interpret \Cref{prop:gamma-and-tilde-gamma} more concretely when $k$ is algebraically closed and we focus on $k$-points.
	Let $p_1, \ldots, p_M$ be the non-factorizable components whose group is contained in~$H$.
	Then, \Cref{prop:gamma-and-tilde-gamma} implies that the $k$-points of $\gamma(H)$ are exactly the points of the form
	\[
		(
			\underbrace{
				x_1, \, \ldots, \, x_M
			}_{\substack{
				\text{coordinates corresponding to}\\
				p_1, \, \ldots, \, p_M
			}}
			\, , \,
			\underbrace{
				0, \, \ldots, \, 0
			}_{\substack{
				\text{coordinates corresponding to}\\
				\text{other non-factorizable components}
			}}
		)
	\]
	where $(x_1, \ldots, x_M)$ is a $k$-point of $\tilde\gamma(H)$.
	On the other hand, the $k$-points of $\pi_H^{-1}(\tilde\gamma(H))$ are those whose coordinates $(x_1, \ldots, x_M)$ form a $k$-point of $\tilde\gamma(H)$, with no additional restriction on the other coordinates besides the fact that the point has to be a $k$-point of $\Spec \, R$.
\end{remark}

\subsubsection{Free factor families and strata.}

We now fix a free factor family $H_1,\ldots,H_k$ of $H$.
Our goal is to relate the stratum associated to $H$ to the strata associated to its factors $H_1, \ldots, H_k$.
The first tool is the following proposition, which rephrases parts of \Cref{prop:carac-factor-free-tensor} geometrically:

\begin{proposition}
	The natural morphism $\Phi : R^{H_1} \otimes \cdots \otimes R^{H_k} \to R^H$ is an isomorphism, i.e., the map $\Spec \, R^H \to \prod_{i=1}^k \Spec \, R^{H_i}$, which we also denote by $\Phi$, is a homeomorphism (even an isomorphism of schemes).
	Moreover, we have the following equality:
	\begin{equation}
		\label{eqn:tilde-gamma-decompose}
		\tilde\gamma(H)
		=
		\Phi^{-1}\!\left(
			\prod_{i=1}^k
				\tilde\gamma(H_i)
		\right).
	\end{equation}
\end{proposition}

\begin{proof}
	That $\Phi$ is an isomorphism follows from points \ref{prop:carac-factor-free-tensor-i} and \ref{prop:carac-factor-free-tensor-ii} of \Cref{prop:carac-factor-free-tensor}.
	\Cref{eqn:tilde-gamma-decompose} follows directly from \iref{prop:carac-factor-free-tensor}{prop:carac-factor-free-tensor-iii} by rephrasing it in terms of closed subsets and taking complements.
\end{proof}

\begin{proposition}
	\label{prop:desc-stratum-product}
	We have the following description of $\gamma(H)$:
	\[
		\gamma(H)
		=
		V(J^*_H)
		\cap
		\bigcap_{i=1}^k
			\pi_{H_i}^{-1}(\tilde\gamma(H_i))
		=
		V(J^*_H)
		\cap
		\bigcap_{i=1}^k
			\pi_{H_i}^{-1}(\pi_{H_i}(\gamma(H_i)))
		.
	\]
\end{proposition}

\begin{proof}
	We compute $\gamma(H)$ in the following way:
	\begin{align*}
		\gamma(H)
		& =
		V(J^*_H) \cap \pi_H^{-1}(\tilde\gamma(H))
		&
		\text{
			by
			\iref{prop:gamma-and-tilde-gamma}{prop:gamma-and-tilde-gamma-iii}
		}
		\\
		& =
		V(J^*_H) \cap
		\pi_H^{-1}\!\left(
			\Phi^{-1}\!\left(
				\prod_{i=1}^k
					\tilde\gamma(H_i)
			\right)
		\right)
		&
		\text{
			by
			\Cref{eqn:tilde-gamma-decompose}
		}.
	\end{align*}
	Define $\pi_{H_i \to H}$ as the dominant map $\Spec\, R^H \to \Spec\, R^{H_i}$ induced by the inclusion $R^{H_i} \hookrightarrow R^H$, so that $\Phi = \pi_{H_1 \to H} \times \cdots \times \pi_{H_k \to H}$.
	The set
	$
		\prod_{i=1}^k
			\tilde\gamma(H_i)
	$
	can be tautologically described as the subset of $\prod_i \Spec \, R^{H_i}$ formed of elements whose projection in $\Spec \, R^{H_i}$ belongs to $\tilde\gamma(H_i)$ for each $i \in \{1, \ldots, k\}$.
	Therefore, the set
	$
		\Phi^{-1}\!\left(
			\prod_{i=1}^k
				\tilde\gamma(H_i)
		\right)
	$
	is the subset of $\Spec \, R^H$ formed of elements~$x$ whose projection in $\Spec \, R^{H_i}$ (namely $\pi_{H_i \to H} (x)$) belongs to $\tilde\gamma(H_i)$ for all $i \in \{1, \ldots, k\}$.
	This yields the equality:
	\[
		\Phi^{-1}\!\left(
			\prod_{i=1}^k
				\tilde\gamma(H_i)
		\right)
		=
		\bigcap_{i=1}^k
			\pi_{H_i \to H}^{-1}(\tilde\gamma(H_i)).
	\]
	Plugging this in the expression of $\gamma(H)$ obtained above, we get:
	\begin{align*}
		\gamma(H)
		& =
		V(J^*_H)
		\cap
		\pi_H^{-1}\!\left(
			\bigcap_{i=1}^k
				\pi_{H_i \to H}^{-1}(\tilde\gamma(H_i))
		\right)
		\\
		& =
		V(J^*_H)
		\cap
		\bigcap_{i=1}^k
			\pi_{H_i}^{-1}(\tilde\gamma(H_i))
		&
		\text{because }
		\pi_{H_i \to H} \circ \pi_H = \pi_{H_i}.
	\end{align*}
	The second equality follows using \iref{prop:gamma-and-tilde-gamma}{prop:gamma-and-tilde-gamma-ii}.
\end{proof}


\subsubsection{Complete description of the spectrum.}

We make the following definition:

\begin{definition}
	\label{defn:factored-splitter}
	A subgroup $H \in \Sub_{G,D}$ is a \emph{factored splitter} if there exists a free factor family of~$H$ of size at least $2$.
\end{definition}




\Cref{prop:desc-stratum-product} then implies the following theorem:

\begin{theorem}
	\label{thm:desc-spectrum}
	Assume that:
	\begin{itemize}
		\item
			Every nontrivial $H \in \Sub_{G,D}$ is either a non-splitter or a factored splitter.
		\item
			For every non-splitter $H \in \Sub_{G,D}$, there is at most one component of group $H$ of each degree.
	\end{itemize}
	Under these hypotheses, we describe $\gamma(H)$ for every $H \in \Sub_{G,D}$:
	\begin{enumerate}[label=(\roman*)]
		\item
			\label{thm:desc-spectrum-i}
			$\gamma(1) = \{\varpi\}$.
		\item
			\label{thm:desc-spectrum-ii}
			If $H$ is a non-splitter, $\gamma(H)$ is the curve $V(A_H) \setminus \{\varpi\}$ from \Cref{thm:curve-in-gamma}.
		\item
			\label{thm:desc-spectrum-iii}
			Otherwise, $H$ is a factored splitter, and we can write $H = H_1 \times \ldots \times H_k$ where the subgroups $H_1,\ldots,H_k$ form a free factor family of non-splitters.
			Then:
			\[
				\gamma(H)
				=
				V(J^*_H)
				\cap
				\bigcap_{i=1}^k
					\pi_{H_i}^{-1}\Big(
						\pi_{H_i}\big(
							V(A_{H_i})
							\setminus
							\{\varpi\}
						\big)
					\Big).
			\]
	\end{enumerate}
	Since, by \Cref{prop:stratification-of-spec}, the strata $\gamma(H)$ cover $\Spec\,R$, this yields a complete description of $\Spec\,R$.
\end{theorem}

\begin{proof}
	Point \ref{thm:desc-spectrum-i} is clear.
	Consider a nontrivial subgroup $H \in \Sub_{G,D}$, and choose a maximal free factor family $H_1,\ldots,H_k$ of $H$.
	For all $i \in \{1,\ldots,k\}$, the subgroup $H_i$ is a non-splitter, as otherwise~$H_i$ is a factored splitter (by the first hypothesis) and we can construct a longer free factor family, contradicting maximality.
	By \Cref{prop:desc-stratum-product}, we can therefore reduce the proof of \ref{thm:desc-spectrum-iii} to the proof of \ref{thm:desc-spectrum-ii}.
	Finally, point \ref{thm:desc-spectrum-ii} follows from \iref{thm:curve-in-gamma}{thm:curve-in-gamma-iv} and from the second hypothesis.
\end{proof}

\subsection{Complete description of the variety of components, in coordinates}
\label{ssn:desc-varcomp-coord}

In this subsection, we rephrase \Cref{thm:desc-spectrum} in terms of coordinates.
We let~$p_1,\ldots,p_N$ be the non-factorizable components and let $d(1), \ldots, d(N)$ be their respective degrees.

We assume for the whole subsection that $k$ is algebraically closed.
As explained in \Cref{subsubsn:affine-embed}, we embed the set $(\Spec\,R)(k)$ of $k$-points of $\Spec\,R$ into the affine space $k^N$.
We denote by $e_1, \ldots, e_N$ the basis elements of~$k^N$, each one corresponding to a non-factorizable component.

\begin{definition}
	\label{defn:point-eh}
	To each subgroup $H \in \Sub_{G,D}$, we associate the following point of $k^N$:
	\[
		e_H
		\eqdef
		\sum_{i \text{ such that } \gen{p_i} \subseteq H}
			e_i.
	\]
\end{definition}

Moreover, we introduce the notion of ``(strict) weighted span'' of a set of points:

\begin{definition}
	\label{defn:weighted-span}
	Let $\{x_1, \ldots, x_n\}$ be a finite set of points of $k^N$, decomposed in the standard basis:
	\[
		x_i
		=
		\sum_{j=1}^N
			\xi_{i,j}
			e_j.
	\]
	The \emph{weighted span} of $\{x_1, \ldots, x_n\}$ is the set:
	\[
		\left \{
			\sum_{j=1}^N
				\sum_{i=1}^n
					\lambda_i^{d(j)}
					\xi_{i,j}
					e_j
			\verti
				(\lambda_1, \ldots, \lambda_n) \in k^n
		\right \}.
	\]
	and the \emph{strict weighted span} of $\{x_1, \ldots, x_n\}$ is its weighted span, minus the weighted span of any proper subset, i.e.:
	\[
		\left \{
			\sum_{j=1}^N
				\sum_{i=1}^n
					\lambda_i^{d(j)}
					\xi_{i,j}
					e_j
			\verti
				(\lambda_1, \ldots, \lambda_n) \in (k^{\times})^n
		\right \}.
	\]
\end{definition}

\begin{example}
	When the degrees $d(i)$ are all equal, the weighted span of $\{x_1, \ldots, x_n\}$ is simply the linear subspace of $k^N$ spanned by $x_1, \ldots, x_n$.
	For a more interesting example, take $N=2, d(1)=1, d(2)=r$.
	The weighted span of the singleton $\{(1,1)\}$ is then the graph of $x \mapsto x^r$ in $k^2$.
\end{example}

We now use this terminology to rephrase \Cref{thm:desc-spectrum} more explicitly:

\begin{theorem}
	\label{thm:desc-spectrum-coords}
	Under the hypotheses of \Cref{thm:desc-spectrum}, the $k$-points of the strata $\gamma(H)$ admit the following description, for each $H \in \Sub_{G,D}$:
	\begin{itemize}
		\item
			The stratum $\gamma(1)$ contains a single $k$-point, namely the origin $(0,\ldots,0) \in k^N$.
		\item
			If $H$ is nontrivial, then we can write $H$ as the product of a free factor family $H_1,\ldots,H_k$ of non-splitters.
			Then, the set of $k$-points of $\gamma(H) \subseteq \Spec\,R$ is, as a subset of $k^N$, the strict weighted span of the points $e_{H_1}, \ldots, e_{H_k}$.
	\end{itemize}
\end{theorem}

\begin{proof}
	The case $H=1$ is clear.
	We consider the case of a nontrivial subgroup $H \in \Sub_{G,D}$.
	The hypotheses imply that $H$ admits a free factor family $H_1, \ldots, H_k$ of non-splitters.
	Let $p_{0,1}, \ldots, p_{0, N(0)}$ be the non-factorizable components whose group is not contained in $H$ and, for each $i \in \{1,\ldots,k\}$, let $p_{i,1}, \ldots, p_{i,N(i)}$ be the non-factorizable components whose group is contained in $H_i$.
	Since $H_1, \ldots, H_k$ is a factor family, there are no other non-factorizable components, and since it is a free family these lists do not overlap.
	We may assume that the $k$-points of $\Spec\,R$, seen as points of $k^N$ (where $N=N(0)+N(1)+\cdots+N(k)$), have their coordinates $x_{i,j}$ (corresponding respectively to the non-factorizable components $p_{i,j}$) ordered in the following way:
	\[
		(
			\underbrace{
				x_{0,1}, \, \ldots, \, x_{0,N(0)}
			}_{\substack{
				\text{non-factorizable components}\\
				\text{of group not contained in } H
			}}
			\, , \,
			\underbrace{
				x_{1,1}, \, \ldots, \, x_{1,N(1)}
			}_{\substack{
				\text{non-factorizable components}\\
				\text{of group contained in } H_1
			}}
			\, , \,
			\ldots
			\, , \,
			\underbrace{
				x_{k,1}, \, \ldots , \, x_{k,N(k)}
			}_{\substack{
				\text{non-factorizable components}\\
				\text{of group contained in } H_k
			}}
		).
	\]
	Note that, for $i \in \{1, \ldots, k\}$, the point $e_{H_i}$ from \Cref{defn:point-eh} has coordinates
	$
		x_{i',j}
		=
		\begin{cases}
			1 & \text{if } i=i' \\
			0 & \text{otherwise.}
		\end{cases}
	$

	The $k$-points of $\gamma(H)$ are, in particular, $k$-points of $V(J^*_H)$, and hence they satisfy $x_{0,1} = \cdots = x_{0,N(0)} = 0$.
	By \Cref{prop:desc-stratum-product}, the only additional condition that they satisfy is that, for each $i \in \{1, \ldots, k\}$, their projection on $k^{N(i)}$ via the morphism $\pi_{H_i}$, which is the point $(x_{i,1}, \ldots, x_{i,N(i)}) \in k^{N(i)}$, must be a $k$-point of $\tilde\gamma(H_i)$.

	Let $i \in \{1, \ldots, k\}$.
	By hypothesis, there is at most one component of group $H_i$ for each degree and each $i \in \{1, \ldots, k\}$.
	Let $d(i,j)$ denote the degree of the non-factorizable component $p_{i,j}$.
	By \iref{thm:curve-in-gamma}{thm:curve-in-gamma-v} and \iref{prop:gamma-and-tilde-gamma}{prop:gamma-and-tilde-gamma-ii}, the $k$-points of~$\tilde\gamma(H_i)$ are then exactly the points of $k^{N(i)}$ of the form $(\lambda_i^{d(i,1)},\ldots,\lambda_i^{d(i,N(i))})$ for some $\lambda_i \in k^\times$.

	Putting everything together, the $k$-points of $\gamma(H)$ are the points of the form:
	\[
		(
			\underbrace{
				0, \, \ldots, \, 0
			}_{N(0)}
			\, , \,
			\underbrace{
				\lambda_1^{d(1,1)}, \,
				\ldots, \,
				\lambda_1^{d(1,N(1))}
			}_{}
			\, , \,
			\ldots
			\, , \,
			\underbrace{
				\lambda_k^{d(k,1)}, \,
				\ldots, \,
				\lambda_k^{d(k,N(k))}
			}_{}
		)
		\quad
		\text{where }
		\lambda_1, \ldots, \lambda_k \in k^\times.
	\]
	The set of such points is exactly the strict weighted span of the points $e_{H_i}$, proving the claim.
\end{proof}

\subsection{An application: the case of symmetric groups}
\label{ssn:symgp}

We now give a concrete example where the variety of components can be described (and even drawn) using \Cref{thm:desc-spectrum-coords}.
In \cite[Chapter~6]{SegThese}, we have focused on the following situation: $G$ is the symmetric group $\Sym_d$ for some $d \geq 2$, the set $D$ is the singleton containing the conjugacy class $c$ of transpositions, and $\xi$ maps $c$ to $1$.
A careful study of the action of braid groups on tuples of transpositions \cite[Theorem~6.2.6]{SegThese} shows that the following properties hold:
\begin{itemize}
	\item
		The non-factorizable components of $\Comp_{\PC}(\Sym_d , \, \{c\} , \, 1)$ are the orbits of tuples $(\tau, \tau)$ where $\tau \in \Sym_d$ is a transposition.
		In particular, there are $d(d-1)/2$ non-factorizable components, all of which have degree $2$.
	\item
		The $D$-generated subgroups $H$ of $\Sym_d$ are all of the form $\Sym_{A_1} \times \cdots \Sym_{A_k}$ for some partition $A_1 \sqcup \ldots \sqcup A_k$ of $\{1,\ldots,d\}$.
		The factors $\Sym_{A_i}$ are shown to form a free factor family of $H$ using \Cref{prop:criterion-factor-family}.
	\item
		If $A$ is a subset of $\{1, \ldots, d\}$, then the subgroup $\Sym_A$ of $\Sym_d$ is a non-splitter, and for each even degree $n \geq 2\card{A}-2$ there is exactly one component of group $\Sym_A$ and of degree $n$.
\end{itemize}

In particular, the hypotheses of \Cref{thm:desc-spectrum}/\ref{thm:desc-spectrum-coords} are satisfied, which leads to a description of the variety of components.
We see it as embedded in the affine space $k^{\frac{d(d-1)}2}$, where we index the coordinates using pairs $(i,j)$ with $1 \leq i < j \leq d$.
If $A$ is a subset of $\{1, \ldots, d\}$, we define:
\[
	e_A
	\eqdef
	\sum_{\substack{
		1 \leq i < j \leq d \\
		i, j \in A
	}}
		e_{i,j}
\]
Then, \Cref{thm:desc-spectrum-coords} implies:

\begin{theorem}
	\label{thm:varcomp-symgp}
	The subset $(\Spec\,R)(k)$ of $k^{\frac{d(d-1)}2}$ is the union of the vector spaces $\Span_k(e_{A_1}, \ldots, e_{A_k})$ over (maximal) families $\{A_1, \ldots, A_k\}$ of disjoint subsets of $\{1, \ldots, d\}$.
\end{theorem}

For the details, we refer to \cite[Section~6.4]{SegThese}.
We now describe the set $V = (\Proj\,R)(k) \simeq (\Spec\,R)(k)/k^\times$ for small values of $d$:
\begin{itemize}
	\item[($d=3$)]
		The set $V$ embeds into $\mathbb{P}^2(k)$: it consists of four points, corresponding to the (lines spanned by the) points $e_{1,2}$, $e_{1,3}$, $e_{2,3}$ and $e_{1,2}+e_{1,3}+e_{2,3}$, i.e., to the subsets of $\{1,2,3\}$ of size $\geq 2$.
	\item[($d=4$)]
		The set $V$ embeds into $\mathbb{P}^5(k)$: it is the union of five points, corresponding to the points $e_A$ for subsets $A \subseteq \{1, \ldots, 4\}$ of size $\geq 3$, and of two lines, corresponding to the two pairs of disjoint subsets of size $2$.
	\item[($d=5$)]
		The set $V$ embeds into $\mathbb{P}^9(k)$: it is the union of six points (for the subsets $A \subseteq \{1, \ldots, 5\}$ of size $\geq 4$), and of ten lines (for the pairs consisting of a subset of size $3$ and its complement).
	\item[($d=6$)]
		The set $(\Proj\,R)(k)$ embeds into $\mathbb{P}^{14}(k)$: it is the (non-disjoint!) union of $22$ points (for the subsets of $\{1, \ldots, 6\}$ of size $\geq 4$), of $10$ lines (for the pairs of disjoint subsets of size $3$) and of $15$ planes (for the triples of disjoint subsets of size $2$).
\end{itemize}

\printbibliography
\end{document}